\documentclass[a4paper,12pt]{article}
\usepackage{graphicx} 
\usepackage{amsmath,amssymb,amsthm}
\usepackage{float} 
\usepackage{subfigure} 
\usepackage[shortlabels]{enumitem}
\usepackage{multirow}
\usepackage[margin=2cm]{geometry}

\newtheorem{theorem}{Theorem}[section]
\newtheorem{corollary}[theorem]{Corollary}
\newtheorem{lemma}[theorem]{Lemma}

\newtheorem{proposition}[theorem]{Proposition}

\theoremstyle{definition}
\newtheorem{definition}[theorem]{Definition}

\theoremstyle{remark}

\newcommand\RR{\mathbb{R}}

\newcommand\al{\alpha}
\newcommand\gm{\gamma}
\newcommand\lm{\lambda}
\renewcommand\th{\theta}

\newcommand\la{\langle}
\newcommand\ra{\rangle}

\newcommand\srg{\mathrm{srg}}

\providecommand{\keywords}[1]
{
    \small
    \textbf{Keywords:} #1
}

\newcommand\Aut{\operatorname{Aut}}
\newcommand\proj{\operatorname{proj}}
\newcommand\Spec{\operatorname{Spec}}

\title{Strongly regular graphs with parameters $(85,14,3,2)$ do not exist}
\author{Sergey Shpectorov and Tianxiao Zhao}
\date{\today}

\begin{document}

\maketitle

\abstract{We investigate the second smallest unresolved feasible set of parameters of 
strongly regular graphs, $(v,k,\lambda,\mu)=(85,14,3,2)$. Using the classification of 
cubic graphs of small degree, we restrict possible local structure of such a graph $G$. 
After that, we exhaustively enumerate possible neighbourhoods of a maximal $3$-clique 
of $G$ and check them against a variety of conditions, including the combinatorial 
ones, coming from $\lm=3$ and $\mu=2$, as well as the linear algebra ones, utilising 
the Euclidean representation of $G$. These conditions yield contradiction in all cases, 
and hence, no $\srg(85,14,3,2)$ exists.}

\bigskip\noindent

\section{Introduction}

In this paper we consider undirected graphs without loops and multiple edges. A 
\emph{strongly regular graph} is a connected regular graph $G$ such that the number 
of common neighbours of two distinct vertices $u,v\in G$ depends only on whether or not 
$u$ and $v$ are adjacent. Four parameters are used to describe the properties of a 
strongly regular graph $G$: the number of vertices, $v$, the valency, $k$, the number 
of neighbours of two adjacent vertices, $\lambda$, and the number of neighbours of 
two non-adjacent vertices, $\mu$. We will use the notation $\srg(v,k,\lm,\mu)$ for any 
strongly regular graphs with these parameters.

The four parameters are not independent; in fact, they satisfy several \emph{feasibility 
conditions}. Parameters satisfying these conditions are called \emph{feasible}. Given a 
feasible parameter set $(v,k,\lm,\mu)$, one can ask whether there is a strongly regular 
graph with such parameters, and if so, how many different $\srg(v,k,\lm,\mu)$ are there 
up to isomorphism. 

The answer to this question varies for different parameter sets. For some feasible 
parameters, the corresponding strongly regular graphs exist and occasionally there may be 
a significant number of non-isomorphic graphs with the same parameters. For other 
feasible parameter sets, no $\srg(v,k,\lm,\mu)$ exist. In other words, feasibility of 
parameters does not guarantee the existence of a strongly regular graph. All feasible 
parameter sets with $v\leq 1300$ are listed in an online catalogue \cite{BrouwerTable} 
maintained by Brouwer. For each parameter set, the catalogue lists the key properties 
of the graph and additional comments describing what is currently known about this case.

We note that the complement graph of a strongly regular graph is also strongly regular, 
as long as it is connected. Hence, strongly regular graphs and their parameter sets 
normally come in pairs. For $v\le 100$, there are only nine unresolved cases, where it 
is not known whether the strongly regular graphs with the given feasible parameter set 
exist. The smallest case is for $v=69$ and the next three unresolved cases have $v=85$. 
In this paper, we resolve one of these three cases. Namely, we prove the following result.

\begin{theorem} \label{mainthm}
There is no strongly regular graph with parameters $(85,14,3,2)$.
\end{theorem}

The complementary array is $(85,70,57,60)$ and, clearly, such strongly regular graphs 
also cannot exist. Previous knowledge about $\srg(85,14,3,2)$ was very limited. 
Paduchikh in \cite{Paduchikh+2009+89+111} investigated how automorphisms of prime order 
could act on such a graph and what would be the fixed subgraph of such an automorphism.

There are several ingredients to our proof of Theorem \ref{mainthm}. First of all, each 
\emph{local subgraph} $G_1(x)$ (i.e., subgraph induced on the neighbourhood of a vertex 
$x$) of $G=\srg(85,14,3,2)$ is a cubic graph on $14$ vertices. Connected cubic graphs 
on at most $14$ vertices have been completely enumerated, see e.g., \cite{MR424607,
BUSSEMAKE1977234}. Going through the list, we determine all possible local graphs, 
using the additional strong condition that two non-adjacent vertices in the local graph 
can have at most one common neighbour. This holds since $\mu=2$. The final list of 
possible local graphs, called \emph{good graphs} below, includes $36$ connected graphs 
and $3$ disconnected graphs, consisting of two components, a $4$-clique and a connected 
cubic graph on $10$ vertices.

One immediate corollary of this is that $G$ contains maximal $3$-cliques. We select one 
such clique, $Q=\{x,y,z\}$, and we do in the computer algebra system GAP \cite{GAP4} a 
complete enumeration of possible subgraphs arising on the set $T$ of vertices adjacent 
to $Q$. Since $|T|=30$, this enumeration is huge and it cannot be done carelessly. We 
represent $T$ as a union of three $12$-vertex \emph{segments}, $S_x=G_1(x)\setminus Q$, 
$S_y=G_1(y)\setminus Q$ and $S_z=G_1(z)\setminus Q$, corresponding to the three vertices 
in $Q$. Each segment is a subgraph of a local graph, and so possible segments can be 
enumerated up to isomorphism, giving us a total of $478$ possible segments. Within $T$, 
two segments intersect in a $2$-vertex set called a \emph{handle} (see Section \ref{segs} 
for all relevant definitions and discussion). A handle can be an edge or a non-edge and 
this leads to a compatibility condition for pairs of segments within $T$. We pre-compute 
the list of $86333$ (ordered) pairs of compatible segments $S_x\cup S_y$ joined at a 
handle, up to isomorphism. These pairs give us the cases into which we split the entire 
calculation. In each case, the calculation goes through four steps: 
\begin{itemize}
\item Step 1: Enumerating possible graph structures on $S_x\cup S_y$.
\item Step 2: Gluing in a third segment $S_z$ in all possible compatible ways and 
enumerating all graph structures on $T=S_x\cup S_y\cup S_z$. A great majority of cases 
are eliminated at this step.
\item Step 3: For each $T$ not eliminated at Step 2, enumerate all possible sets $C$ of 
additional neighbours of a fixed vertex $t\in X=S_y\cap S_z$.
\item Step 4: Enumerate all possible graph structures on $C$ and hence on the entire
$T\cup C$, achieving elimination in all cases.
\end{itemize}

In addition to purely combinatorial arguments involving the parameters of $G$, we rely 
for elimination on the linear algebra conditions coming from the \emph{Euclidean 
representation} of $G$. It is well-known that a strongly regular graph can be realised 
as a set of unit vectors in an eigenspace of its adjacency matrix. The value of the dot 
product of two unit vectors corresponding to vertices $u$ and $v$ is given in the 
so-called \emph{cosine sequence} of $G$, and it depends only on the distance between $u$ 
and $v$ in $G$. The eigenvalues of the adjacency matrix of $G$ and their multiplicities 
can be computed from the parameters of $G$. For $G=\srg(85,14,3,2)$, the adjacency 
matrix has eigenvalues $k=14$, $4$, and $-3$, with respective multiplicities $1$, $34$, 
and $50$. For our calculation, we selected the unit vector realisation in the eigenspace 
$E$ of dimension $34$ corresponding to the eigenvalue $4$. The cosine sequence for this 
Euclidean representation is $w_0=1$, $w_1=\frac{2}{7}$, and $w_2=-\frac{1}{14}$. That 
is, adjacent vertices lead to the dot product value $\frac{2}{7}$ and the non-adjacent 
vertices lead to the value $-\frac{1}{14}$. 

Once we know all edges on a subset $X$ of $G$, such as, say, $T$ or $T\cup C$ (or a 
subset of these sets), we can create the Gram matrix corresponding to $X$. Since the dot 
product is positive definite on $E$, the Gram matrix of every subset of $G$ must be 
semi-positive definite, i.e., it cannot have negative eigenvalues. The strength of this 
linear algebra condition grows with the size of the subset $X$ we consider. For example, 
none of the possible sets $S_x\cup S_y$ (of cardinality $22$) is eliminated by this 
condition. However, for $T=S_x\cup S_y\cup S_z$ (whose size is $30$), this condition is 
very powerful and it eliminates a large majority of all possible configurations of edges, 
depending on the specific case. Clearly, we want to eliminate each configuration as early 
as possible, so we build up $T$ vertex by vertex and check semi-positive definiteness 
along the way.

Also, note that the rank of the Gram matrix cannot exceed $\dim E=34$. Hence, at Step 4, 
where the size of $T\cup C$ grows to $38>34$, the rank consideration can be applied with 
a devastating effect, eventually eliminating all configurations.

While the overall idea of our four-step enumeration looks quite simple and 
straightforward, it was found as a result of much experimentation. Even with our approach 
above, the total enumeration involves an astronomical number of possible configurations, 
and so the enumeration would not have been possible without a very effective code in GAP. 
Because of this, we also devote much attention below to the exact algorithmic details, 
including some key data structures and even some code. We do not include the entire code 
we created and used, due to its length, but it is available on GitHub \cite{code}.

\medskip
Finally, let us describe the contents of the paper section by section. In Section 2, we 
provide the background information on strongly regular graphs and their Euclidean 
representations. In Section 3, we identify the possible local subgraphs of 
$G=\srg(85,14,3,2)$, starting from the known lists of cubic graphs of small size. In 
Section 4, we discuss the concepts of segments and handles, which are the building blocks 
for our set $T$. We determine the complete list of possible segments and classify them 
according to their type. Gluing of segments over the common handles is described in 
Section 5, where we develop the group-theoretic methods allowing us to avoid repetitions. 
This leads to finding the list of $86333$ pairs of segments glued over a common handle, 
which constitute the cases into which we split the calculation. We also describe in this 
section the main idea of our approach, focussing on the set $T=S_x\cup S_y\cup S_z$. Not 
every edge in $T$ is contained in one of the three segments. Hence, we study the edges 
that cut across segments, and we show that such edges form a matching between the subsets 
in the two segments, called the \emph{cores}. This results informs the enumeration method 
we select, using \emph{enumeration trees}. In Section 6, we provide some data and 
algorithmic details: ordering of the vertices in a segment and the resulting quad type, 
which accomplishes a finer classification of segments. We also discuss in this section 
the overall organisation of Steps 1 and 2. 

By the end of Step 2, we have already eliminated a large majority of all cases. 
However, even a tiny percentage of survivors, leads to a very large number of cases, 
for which we need to do further steps. In Section 7, we describe properties of the 
additional vertices we add, namely to the additional neighbours of a vertex $t\in 
X=S_y\cap S_z$. We identify each additional vertex with its set of neighbours in $T$, 
discuss compatibility of additional vertices, and provide the details of the recursive 
enumeration we do at Step 3. In Section 8, we similarly discuss the details of Step 4, 
where we have in hand a complete set $C$ of neighbours of $t$ and we enumerate the 
possible edges on $C$ and achieve the final elimination. In Section 9, we describe an 
additional idea, based on the careful selection of the $3$-clique $Q=\{x,y,x\}$, 
identifying our set $T$. This additional idea leads to a significant overall reduction 
in the number of cases we need to consider, and it also leads to a more simple algorithm. 
Section 10 contains brief concluding remarks. The paper has three appendices. Appendix A 
describes the enumeration trees we use at Steps 1 and 2. Appendix B deals with the 
details of the LDLT algorithm we use to verify semi-positive definiteness as we add 
vertices one by one. Finally, in Appendix C we describe and justify the method we use to 
compute projections to $E$, which we use extensively at Steps 3 and 4. 

\medskip
The enumeration was carried out in parallel on $96$ cores in four servers in the School 
of Mathematics at the University of Birmingham. We especially thank David Craven, who 
manages these servers and who tolerated our $96$ copies of GAP, running continuously 
from November 2023 till January 2025.

\section{Preliminaries} 

\subsection{Graphs}

In this section we will briefly introduce some background concepts and facts.

By a graph we mean a simple graph, without loops or multiple edges. We will identify a graph 
$G$ with its vertex set and will use $E(G)$ to denote its edge set. We will only deal with 
finite graphs, i.e., $|G|$ (and hence also $|E(G)|$) will be finite. We call $v=|G|$ the 
\emph{order} of the graph $G$.

For $x,y\in G$, we write $x\sim y$ (respectively, $x\not\sim y$) to indicate that $x$ and 
$y$ are adjacent (respectively, non-adjacent). A \emph{path} from $x$ to $y$ is a sequence 
$p=(x_0,x_1,\ldots,x_n)$ of vertices, where $x=x_0$, $y=x_n$ and $x_{i-1}$ is adjacent to 
$x_i$ for all $i=1,2,\ldots,n$. We call $n$ the \emph{length} of the path $p$. We will 
typically assume that the graph is connected, that is, any two vertices of it are connected 
by a path. Length of the shortest path between $x$ and $y$ is known as the \emph{distance} 
between $x$ and $y$. The largest distance between vertices of $G$ is called the 
\emph{diameter} of $G$. For $i\geq 1$, by $G_i(x)$ we mean the set of vertices of $G$ at 
distance $i$ from $x$. In particular, $G_1(x)$ is the set of neighbours of $x$ in $G$. Recall 
that $G$ is called $k$-\emph{regular} if $|G_1(x)|=k$ for every $x\in V(G)$. Then we call $k$ 
the \emph{degree} of the regular graph $G$. 

For $X\subset G$, we will similarly identify the subset $X$ with the induced subgraph on $X$. 
This is the subgraph that includes all edges with both ends in $X$. In particular, we will 
often refer to $G_1(x)$ as the \emph{local subgraph} of $G$. 

Consider a graph $G=\left\{u_1,u_2,\dots,u_v\right\}$. The \textit{adjacency matrix} $A=A(G)$ 
is the square matrix of size $v$, whose entries satisfy:
\begin{equation*}
a_{ij}=
\left\{
\begin{array}{rl}
1, & \text{if }u_i\sim u_j,\\    
0, & \text{if }u_i\not\sim u_j.
\end{array}
\right.
\end{equation*}
The \textit{spectrum} of the graph $G$ is the spectrum of its adjacency matrix $A$. In other 
words, if $A$ has eigenvalues $\lm_1,\lm_2,\dots,\lm_s$ with multiplicities $m_1,m_2,\dots,
m_s$, then the multiset $\Spec(G)=\{\lm_1^{m_1},\lm_2^{m_2},\dots,\lm_s^{m_s}\}$ is the 
spectrum. Note that $A$ is symmetric and hence all its eigenvalues $\lm_i$ are real. 
Furthermore, $A$ is semisimple (diagonalisable), that is, $\RR^v$ decomposes as the direct 
sum of the eigenspaces of $A$ and hence $\sum_{i=1}^sm_i=v$.

We also note the following standard fact about symmetric real matrices.

\begin{lemma}\label{Orthogonal eigenspaces}
Distinct eigenspaces of $A=A(G)$ are orthogonal with respect to the dot product on $\RR^v$.
\end{lemma}

Therefore, the above decomposition of $\RR^v$ as a direct sum of eigenspaces of $A$ is 
orthogonal.

If $G$ is $k$-regular then each row of $A$ has exactly $k$ ones. Hence $k$ is an eigenvalue 
of $A$. Furthermore, if $G$ is connected then $k$ has multiplicity 1. The corresponding 
eigenspace is spanned by the all-one vector.

\medskip
We conclude this section with the following definition. An \emph{automorphism} of a graph $G$ 
is a permutation of $G$ that preserves adjacency. We denote by $\Aut(G)$ the group of all 
automorphisms of $G$.

\subsection{Strongly regular graphs}

A \textit{strongly regular graph} is a connected regular graph, for which there exist 
non-negative integers $\lm$ and $\mu$, such that any two adjacent vertices have exactly $\lm$ 
common neighbours and any two non-adjacent vertices have exactly $\mu$ common neighbours. We 
will write $\srg(v,k,\lm,\mu)$ for a strongly regular graph on $v$ vertices, of degree $k$, 
and with parameters $\lm$ and $\mu$, as above. 

We follow \cite{godsil-2001} for the basic results on strongly regular graphs.

\begin{theorem}
For an $\srg(v,k,\lm,\mu)$, let $A$ be its adjacency matrix, $I$ be the identity matrix and 
$J$ be the all-one matrix, both of the same order as $A$. Then the following relations hold:
$$
AJ=kJ\mbox{ and }A^2+(\mu-\lm)A+(\mu-k)I=\mu J.
$$
\end{theorem}

As an application of these equations, we have the following theorem.

\begin{theorem}\label{spectrum of srg}
If $G=\srg(v,k,\lm,\mu)$ then its spectrum is $\left\{ k^1,r^f,s^g\right\}$, where $r,s,f,g$ 
are given by:
\begin{align*}
r&=\tfrac{1}{2}\left(\lm-\mu+\sqrt{(\lm-\mu)^2+4(k-\mu)}\right),\\
s&=\tfrac{1}{2}\left(\lm-\mu-\sqrt{(\lm-\mu)^2+4(k-\mu)}\right),\\
f&=\tfrac{1}{2}\left(v-1-\frac{2k+(v-1)(\lm-\mu)}{\sqrt{(\lm-\mu)^2+4(k-\mu)}}\right),\\
g&=\tfrac{1}{2}\left(v-1+\frac{2k+(v-1)(\lm-\mu)}{\sqrt{(\lm-\mu)^2+4(k-\mu)}}\right).
\end{align*}
\end{theorem}

Let us now see what this translates to for a possible $\srg(85,14,3,2)$.

\begin{corollary} \label{eigenvalues and multiplicities}
If $G=\srg(85,14,3,2)$ then $A(G)$ has eigenvalues $14$, $4$, and $-3$ with multiplicities $1$, 
$34$, and $50$, respectively.
\end{corollary}

This is obtained by plugging the values of $(v,k,\lm,\mu)=(85,14,3,2)$ into the formulae from 
Theorem \ref{spectrum of srg}.

\subsection{Euclidean representation}

Now we introduce the \textit{Euclidean representation} of a strongly regular graph, which will 
be the main tool that we will use to eliminate cases in this project. The theory is mostly based 
on the book of Godsil \cite{ACGodsil}. There it is developed for arbitrary distance-regular 
graphs, so we adjust it for our case: strongly regular graphs are distance-regular graphs of 
diameter $2$. We also slightly alter the notation from the book.

As above, let $G=\left\{u_1,u_2,\dots,u_v\right\}$ be a strongly regular graph with parameters 
$(v,k,\lm,\mu)$. We will identify $G$ with the standard basis of $U=\RR^v$. 

Recall that $k$, $r$ and $s$ are the eigenvalues of the adjacency matrix $A=A(G)$ of $G$, of 
multiplicity $1$, $f$ and $g$, respectively. It follows from Theorem \ref{spectrum of srg} and 
the discussion around Lemma \ref{Orthogonal eigenspaces} that we have the following orthogonal 
decomposition:
$$
U=U_k\oplus U_r\oplus U_s,
$$
where $U_\th$ is the $\th$-eigenspace of $A$ for each $\th\in\{k,r,s\}$. As we already mentioned, 
since $G$ is connected, the $1$-dimensional eigenspace $U_k$ is spanned by the all-one vector. 
For $\th\in\{k,r,s\}$, consider the orthogonal projection $p_\th:U\rightarrow U_\th$. Then we 
have the following result (see Lemma 1.2 in Chapter 13 of \cite{ACGodsil}).

\begin{theorem} \label{depends on distance}
For $u_i,u_j\in V(G)$, the value of $p_\th(u_i)\cdot p_\th(u_j)$ depends only on the distance 
between $u_i$ and $u_j$. 
\end{theorem}

We will represent the vertices of $G$ by their images under $p_\th$, but scaled to have length 
$1$. That is, the Euclidean representation of $G$ with respect to $\th$ represents every vertex 
$u_i$ by the unit vector $e_i:=\frac{1}{|p_\th(u_i)|}p_\th(u_i)$ from the subspace $U_\th$. If 
the distance between $u_i$ and $u_j$ is $m\in\{0,1,2\}$ then we have, by Theorem \ref{depends on 
distance}, that 
$$
w_m:=e_i\cdot e_j=\frac{p_\th(u_i)\cdot p_\th(u_j)}{|p_\th(u_i)||p_\th(u_j)|}
$$
is a function of $m$ alone. The values $w_m$, $m\in\{0,1,2\}$, are known as the \emph{cosine 
sequence} of $G$. We can furthermore claim the following exact formulae for the values of $w_m$.

\begin{theorem}[\cite{BrouwerDRGbook}, Section 4.1B] \label{cosseq for srg}
We have that $w_0=1$, $w_1=\frac{\th}{k}$ and $w_2=\frac{\th^2-\lm\th-k}{k(k-\lm-1)}$.    
\end{theorem}

We now specialise this to the case of $\srg(85,14,3,2)$. First of all, we choose $\th=4$ (c.f. 
Corollary \ref{eigenvalues and multiplicities}), so that our representation is in the eigenspace 
$E:=U_4$ of dimension $34$. From now on, this is our fixed choice.

\begin{corollary} \label{cosine sequence}
For $G=\srg(85,14,3,2)$ and $\th=4$, the cosine sequence is $w_0=1$, $w_1=\frac{2}{7}$ and 
$w_2=-\frac{1}{14}$.
\end{corollary}

Hence, when the vertices $u_i$ and $u_j$ are adjacent, we have that $e_i\cdot e_j=\frac{2}{7}$ 
and, when $u_i$ and $u_j$ are non-adjacent, we have that $e_i\cdot e_j=-\frac{1}{14}$. Note that 
the unit vectors $\{e_1,e_2,\ldots,e_v\}$ span $E$. Indeed, this follows from the fact that 
$\{u_1,u_2,\ldots,u_v\}$ is a basis of $U=\RR^v$.

From this point on, $G=\srg(85,14,3,2)$ and it is realised by the unit vectors $e_i$ in the 
$34$-dimensional Euclidean space $E$ as above. Our goal is to obtain a contradiction, showing 
non-existence of $G$, by considering various sets of vectors $e_i$. If $X$ is such a set, then 
the Gram matrix corresponding to $X$ must be semi-positive definite, i.e., it cannot have 
negative eigenvalues. This is because the dot product on $E\subset U$ is positive definite.

\section{Local structure of $G$} 

Let us focus on the local subgraph $G_1(x)$ of $G$, induced by the $14$ neighbours of a vertex 
$x\in G$. Since $\lm=3$, the subgraph $G_1(x)$ is cubic, i.e., of degree $3$. Connected cubic 
graphs on at most $14$ have been enumerated, see e.g. \cite{MR424607,BUSSEMAKE1977234}. The 
complete list of all these graphs can be found on the internet and examined via, say, the 
computer algebra system GAP \cite{GAP4}. Since at this point we cannot assume that the local 
subgraphs $G_1(x)$ are connected, we should also add disconnected cubic graphs on $14$ vertices, 
as unions of smaller connected cubic graphs. In total, there are, up to isomorphism, $509$ 
connected cubic graphs on $14$ vertices and $31$ disconnected ones. 

\begin{lemma}\label{localgraph1}
Any two non-adjacent vertices $y,z\in G_1(x)$ have at most one common neighbour in $G_1(x)$.
\end{lemma}

\begin{proof}
Since $\mu=2$, any two non-adjacent vertices $y,z\in G_1(x)$ have two common neighbours in $G$. 
One of them is $x$, hence $y$ and $z$ have at most one common neighbour in $G_1(x)$.
\end{proof}

It turns out this simple lemma allows us to reduce significantly the number of possible graphs 
$G_1(x)$. Going through the list of all $540$ cubic graphs on $14$ vectices, we discover that 
only $39$ of them, referred to in this paper as the \emph{good} cubic graphs, satisfy the 
condition from Lemma \ref{localgraph1}. Hence we have the following.

\begin{proposition}
The local subgraph $G_1(x)$, for $x\in G$, is isomorphic to one of the $39$ good cubic graphs.
\end{proposition}

Out of these graphs, $36$ are connected and $3$ are disconnected, namely, they are unions of a 
$4$-clique and a connected trivalent graph $T_i$ of order $10$, $i=1,2,3$, shown in Figure 
\ref{cubic graphs}. Note that $T_3$ is the Petersen graph.
\begin{figure}[H]
\centering  
\subfigure[$T_1$]{
\label{T_1 graph}
\includegraphics[width=0.3\textwidth]{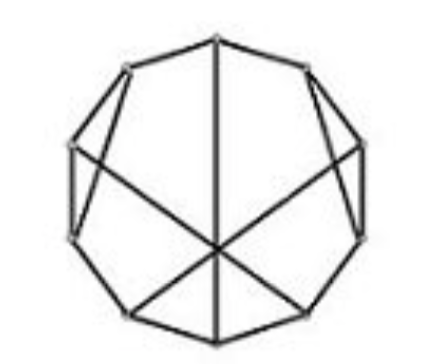}}
\subfigure[$T_2$]{
\label{T_2 graph}
\includegraphics[width=0.3\textwidth]{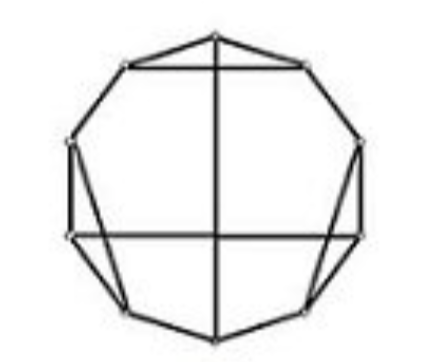}}
\subfigure[$T_3$]{
\label{T_3 graph}
\includegraphics[width=0.3\textwidth]{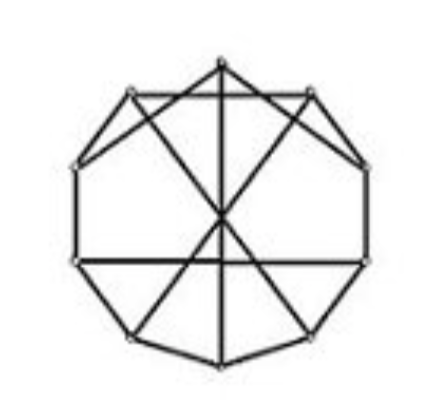}}
\caption{Relevant cubic graphs of order 10}
\label{cubic graphs}
\end{figure}

Our method later in the text is based on considering on a neighbourhood of a maximal $3$-clique 
in $G$. We can derive the existence of such maximal cliques from the enumeration above.

\begin{proposition}\label{maximal 3-cliques exist}
Maximal $3$-cliques exist in $G$.
\end{proposition}

\begin{proof}
For a vertex $x\in G$, maximal $3$-cliques containing $x$ correspond to edges in $G_1(x)$ not 
contained (within $G_1(x)$) in a $3$-clique. Hence the existence of maximal $3$-cliques can be 
established by simply going through the list of $39$ possible local graphs. 
\end{proof}

Alternatively, it is easy to show that if $y\in G_1(x)$ is not contained in a component of size 
$4$ of $G_1(x)$ then its local graph in $G_1(x)$ is either a $3$-coclique or a union of an edge 
and an isolated vertex. Hence every such edge $xy$ is contained in a maximal $3$-clique in $G$. 
We can formulate this as follows, specialising Proposition \ref{maximal 3-cliques exist}.

\begin{proposition}\label{3-clique}
Every edge $xy\in E(G)$ that is not contained in a $5$-clique, is contained in a maximal 
$3$-clique.
\end{proposition}

In turn, this implies the following.

\begin{corollary}\label{Maximal 3-clique}
Every vertex $x\in G$ is contained in a maximal $3$-clique.
\end{corollary}

\section{Segments and handles} \label{segs}

Let $Q$ be a maximal $3$-clique in $G$. For each $x\in Q$, we have that $yz=Q\setminus\{x\}$ is 
an edge in the local subgraph $G_1(x)$. We now focus on the remaining twelve vertices in 
$G_1(x)\setminus yz$. Since $Q$ is maximal, the edge $xy$ is not contained in a $3$-clique in the 
good cubic graph $G_1(x)$. 

\begin{definition}\label{DefSegment}
Suppose that $H$ is a good cubic graph and $yz$ is an edge in $H$ not contained in a $3$-clique. 
Let $S$ be the subgraph of $H$ induced on $H\setminus yz$. We call $S$ the \emph{segment} 
corresponding to $H$ and $yz$.
\end{definition}

Since $yz$ is not contained in a $3$-clique, $y$ and $z$ have no common neighbours in $H$. Let 
$y_1$ and $y_2$ be the neighbours of $y$ in $H$, other than $z$, and symmetrically, let $z_1$ and 
$z_2$ be the neighbours of $z$ in $H$, other than $y$. Clearly, the pairs $\{y_1,y_2\}$ and 
$\{z_1,z_2\}$ are disjoint and contained in $S$.

\begin{definition}
We call the pairs $\{y_1,y_2\}$ and $\{z_1,z_2\}$ the \emph{handles} of the segment $S$.
\end{definition}

We note that the segment is not just the graph $S$, but rather the triple consisting of $S$ and 
the two handles $\{y_1,y_2\}$ and $\{z_1,z_2\}$. Note also that while $\{y_1,y_2,z_1,z_2\}$ can 
be identified within $S$ as the set of all vertices of $S$ having only two neighbours, this set 
can potentially be split into a union of two handles in more than one way. Hence we need to view 
a segment as a triple, expressly indicating the handles. The order of the two handles is less 
important, but naturally, they are kept in a computer in some order. When we talk below about 
isomorphisms of segments $S$ and $S'$, we mean graph isomorphisms $S\to S'$ mapping the first 
and second handles of $S$ to the first and second handles of $S'$, respectively.

We now turn to enumeration of segments. The good graph $H$ can be easily recovered from its 
segment $S$. Indeed, we just need to add to $S$ a new vertex $y$ adjacent to the two vertices in 
the first handle $\{y_1,y_2\}$ and a second new vertex $z$ adjacent to $y$ and the vertices in 
the second handle $\{z_1,z_2\}$. In particular, every segment $S$ arises from a unique good 
cubic graph $H$. 

This allows for a very efficient enumeration of segments. We go through the list of the $39$ good 
cubic graphs $H$. In each $H$, we determine the orbits under $\Aut(H)$ on the set of ordered 
edges of $H$. Segments are isomorphic if and only if the come from the same orbit. Hence we 
obtain a complete list of segments by choosing one representative $yz$ of each orbit and 
constructing the segment $H\setminus yz$. 

Before we report the results of this calculation, we need to discuss the \emph{types} of handles 
and segments. A handle is a pair of vertices and, naturally, these two vertices can be adjacent 
or non-adjacent. Hence the handles are classified into \emph{edges} and \emph{non-edges}. 
Correspondingly, segments can be of four types depending on their first and second handle: (1) 
edge and edge; (2) edge and non-edge; (3) non-edge and edge; and (4) non-edge and non-edge. 
Clearly, for every segment of type (2), we obtain a segment of type (3) by switching the handles. 
Types (1) and (4) are invariant under this operation.

\begin{proposition}
There are in total $478$ segments up to isomorphism. Among them there are $19$, $78$, $78$, $303$ 
segments of types $(1)$, $(2)$, $(3)$, and $(4)$, respectively.
\end{proposition}

We now introduce an additional structure on a segment. First, we need the following.

\begin{lemma}\label{segnonadj}
Let $S$ be a segment and $Y=\{y_1,y_2\}$ and $Z=\{z_1,z_2\}$ be its two handles. Then $y_i$ and 
$z_j$ are non-adjacent for all $i,j=1,2$.
\end{lemma}

\begin{proof}
Let $H$ be the good cubic graph that $S$ is obtained from, and $yz$ be the edge removed from $H$. 
We assume that $y_1$ and $y_2$ are adjacent to $y$ while $z_1$ and $z_2$ are adjacent to $z$.

Now we can prove our claim: if $y_i$ and $z_j$ are adjacent then, in $H$, the vertices $y_i$ and 
$z$ have two common neighbours, $y$ and $z_j$. This contradicts the condition from Lemma 
\ref{localgraph1} that we imposed on good cubic graphs.
\end{proof}

We will need the following concept.

\begin{definition}
Let $S$ be a segment and $Y$ and $Z$ be its handles. The \emph{core} of $S$ with respect to the 
handle $Y$ is the set of vertices in $S\setminus(Y\cup Z)$ that are not adjacent to a vertex of 
$Y$. 
\end{definition}

The size of core depends on the type of the handle.

\begin{lemma}\label{SizeOfCore}
If $Y=\{y_1,y_2\}$ is a handle of a segment $S$ then $y_1$ and $y_2$ have no common neighbours in 
$S$. In particular, the core of $S$ with respect to $Y$ consists of six vertices, if $Y$ is an 
edge, and the core consists of four vertices, if $Y$ is a non-edge.
\end{lemma}

\begin{proof}
Let $H$ be the good cubic graph from which $S$ was constructed by removing an edge $yz$. We 
assume that $y$ is adjacent to $y_1$ and $y_2$. Suppose that $y_1$ and $y_2$ have a common 
neighbour $s$ in $S$. Then $s$ is not adjacent to $y$ and at the same time, $s$ and $y$ have two 
common neighbours, $y_1$ and $y_2$. This is a contradiction since $H$ is a good cubic graph. 
Thus, $y_1$ and $y_2$ have no joint neighbours in $S$.

If $Y$ is an edge then $y_1$ has only one neighbour in $S$, other than $y_2$, and symmetrically, 
$y_2$ has only one neighbour in $S$, other than $y_1$. Note that, by Lemma \ref{segnonadj}, 
those neighbours are not in the second handle $Z$. Hence the core with respect to $Y$ is 
obtained by removing from $S$ the vertices in $Y$, $Z$ and the two neighbours of $Y$, leaving 
six vertices in the core.

If $Y$ is a non-edge then each of $y_1$ and $y_2$ has two further neighbours in $S$ and so the 
core is obtained in this case by removing $Y$, $Z$ and the four neighbours of $Y$. Hence the 
core is of size four.
\end{proof}

This allows us to attach numerical labels to the segments; namely, each segment $S$ will be 
labelled with a pair from $\{6,4\}\times\{6,4\}$, indicating the size of the cores with respect 
to the two handles. That is, type (1) above (edge and edge) becomes type $(6,6)$; type (2) 
(edge and non-edge) becomes $(6,4)$; type (3) (non-edge and edge) becomes type $(4,6)$; and 
finally, type (4) (non-edge and non-edge) becomes type $(4,4)$.

In fact, we can classify segments even more finely. For a segment $S$ with handles $X$ and $Y$, 
let the quadruple $(n_S,r_S,l_S,b_S)$ be defined as follows: $n_S=|S_0\setminus(C_X\cup C_Y)|$, 
$r_S=|C_Y\setminus C_X|$, $l_S=|C_X\setminus C_Y|$, and $b_S=|C_X\cap C_Y|$, where 
$S_0=S\setminus(X\cup Y)$, $C_X$ is the core of $S$ with respect to $X$, and $C_Y$ is the core 
of $S$ with respect to $Y$. Clearly, $n_S+r_S+l_S+b_S=|S_0|=8$, $l_S+b_S=|C_X|$, and
$r_S+b_S=|C_Y|$. So this quadruple can serve as a finer invariant, and we call it the 
\emph{quad type} of $S$.

The table below shows the number of segments of each possible quad type.
\begin{table}[H]
\centering
\begin{tabular}{|c|c|c|}
\hline
        type           &  quad type  & number of segments \\ \hline
\multirow{3}{*}{(6,6)} & $(2,0,0,6)$ & 0 \\ \cline{2-3} 
                       & $(1,1,1,5)$ & 5 \\ \cline{2-3} 
                       & $(0,2,2,4)$ & 14 \\ \hline
\multirow{3}{*}{(6,4)} & $(2,0,2,4)$ & 9 \\ \cline{2-3} 
                       & $(1,1,3,3)$ & 35 \\ \cline{2-3} 
                       & $(0,2,4,2)$ & 34 \\ \hline
\multirow{5}{*}{(4,4)} & $(4,0,0,4)$ & 4 \\ \cline{2-3} 
                       & $(3,1,1,3)$ & 23 \\ \cline{2-3} 
                       & $(2,2,2,2)$ & 146 \\ \cline{2-3} 
                       & $(1,3,3,1)$ & 102 \\ \cline{2-3} 
                       & $(0,4,4,0)$ & 28 \\ \hline
\end{tabular}
\caption{Quad types}
\label{quad type}
\end{table}

Note that there happens to be no segments of quad type $(2,0,0,6)$. Note also that we skipped 
the type $(4,6)$, which clearly has the dual (switching $r_S$ and $l_S)$ quad type 
distribution, compared to the type $(6,4)$.

\section{Gluing segments}

The idea of our approach is to investigate the union of three local subgraphs $G_1(x)$, $x\in 
Q$, where $Q$ is, as above, a maximal $3$-clique in $G$. Once we identify the graph induced on 
this set, we can produce the corresponding Gram matrix (since we assume that $G$ is realised 
as a set of unit vectors in $E$) and check this matrix for semi-positive definiteness hoping 
to eliminate in this way a great majority of all possibilities.

Before we discuss how to glue local subgraphs together, let us make a slight adjustment. 
Namely, we will show that it suffices to glue together the three segments arising from $Q$.

We will need the following result.

\begin{lemma} \label{LinearSum}
For $x\in G$, we have that $\sum_{y\in G_1(x)}y=4x$.
\end{lemma}
    
\begin{proof}
Let $w=\sum_{y\in G_1(x)}y$. Using Corollary \ref{cosine sequence}, we have that $w\cdot w
=14(1+3(\frac{2}{7})+10(-\frac{1}{14}))=14+12-10=16$ and $w\cdot x=x\cdot w=14(\frac{2}{7})
=4$. In the first calculation, we used that every $y\in G_1(x)$ has $3$ neighbours and $10$ 
non-neighbours in $G_1(x)$. 

Hence, for $t=w-4x$, we have that $t\cdot t=(w-4x)\cdot(w-4x)=w\cdot w-4w\cdot x-4x\cdot w
+16x\cdot x=16-16-16+16=0$. This means that $t=0$, that is, $w=4x$, as claimed.
\end{proof}

As we already stated, we would like to recover the induced graph structure on the union 
$\cup_{x\in Q}G_1(x)$ for a maximal $3$-clique $Q$. Clearly, this set is the union of $Q$ 
and the three segments $S_x:=G_1(x)\setminus Q$, $x\in Q$. Let $s_x:=\sum_{u\in S_x}u$. 
Then we have the following vector equations.

\begin{lemma}\label{ABClindep}
If $Q=\{x,y,z\}$ then $x=\frac{1}{10}(3s_x+s_y+s_z)$, $y=\frac{1}{10}(s_x+3s_y+s_z)$, and 
$z=\frac{1}{10}(s_x+s_y+3s_z)$.
\end{lemma}

\begin{proof}
By Lemma \ref{LinearSum}, we have:
\begin{align*}
s_x+y+z&=4x\\
s_y+x+z&=4y\\
s_z+x+y&=4z 
\end{align*}
Solving this linear system for $x$, $y$ and $z$ yields the claim.
\end{proof}

According to this lemma, the span of $T:=\cup_{x\in Q}S_x$ includes $Q$ and so the rank of the 
Gram matrix on $T\cup Q$ is the same as the rank of the Gram matrix on $T$. We will therefore 
work with just $T$, ignoring the three vertices from $Q$.

Next we discuss how we can recover the graph structure on $T$, which includes how the segments 
$S_x$, $x\in Q$, intersect, which edges $T$ inherits from these segments, and which edges in 
$T$ are extra.

We start with the intersection of segments. 

\begin{lemma} \label{intersection}
For $x,y\in Q$, $x\neq y$, we have that $Z:=S_x\cap S_y$ is a handle in both segments $S_x$ 
and $S_y$.
\end{lemma}

\begin{proof}
If $z$ is the third vertex in $Q$ then $S_x$ is obtained from $G_1(x)$ by removing the edge 
$yz$. Hence $S_x\cap G_1(y)$ is a handle in $S_x$. It remains to see that $S_x\cap G_1(y)=Z$. 
On the one hand, $Z=S_x\cap S_y$ is clearly contained in $S_x\cap G_1(y)$. On the other hand, 
the difference between $S_y$ and $G_1(y)$ is the edge $xz$. Both $x$ and $z$ are not 
contained in $S_x$, and hence, indeed, $S_x\cap G_1(y)=S_x\cap S_y=Z$. 

Symmetrically, $Z$ is also a handle in $S_y$.
\end{proof}

Let us now fix the notation we will use in the remainder of the paper. Let $Q=\{x,y,z\}$ and 
we set $Z:=S_x\cap S_y$, as above, and symmetrically, $Y:=S_x\cap S_z$ and $X:=S_y\cap S_z$. 
This is shown in Figure \ref{three segments}.
\begin{figure}[ht] 
\centering
\includegraphics[width=8cm]{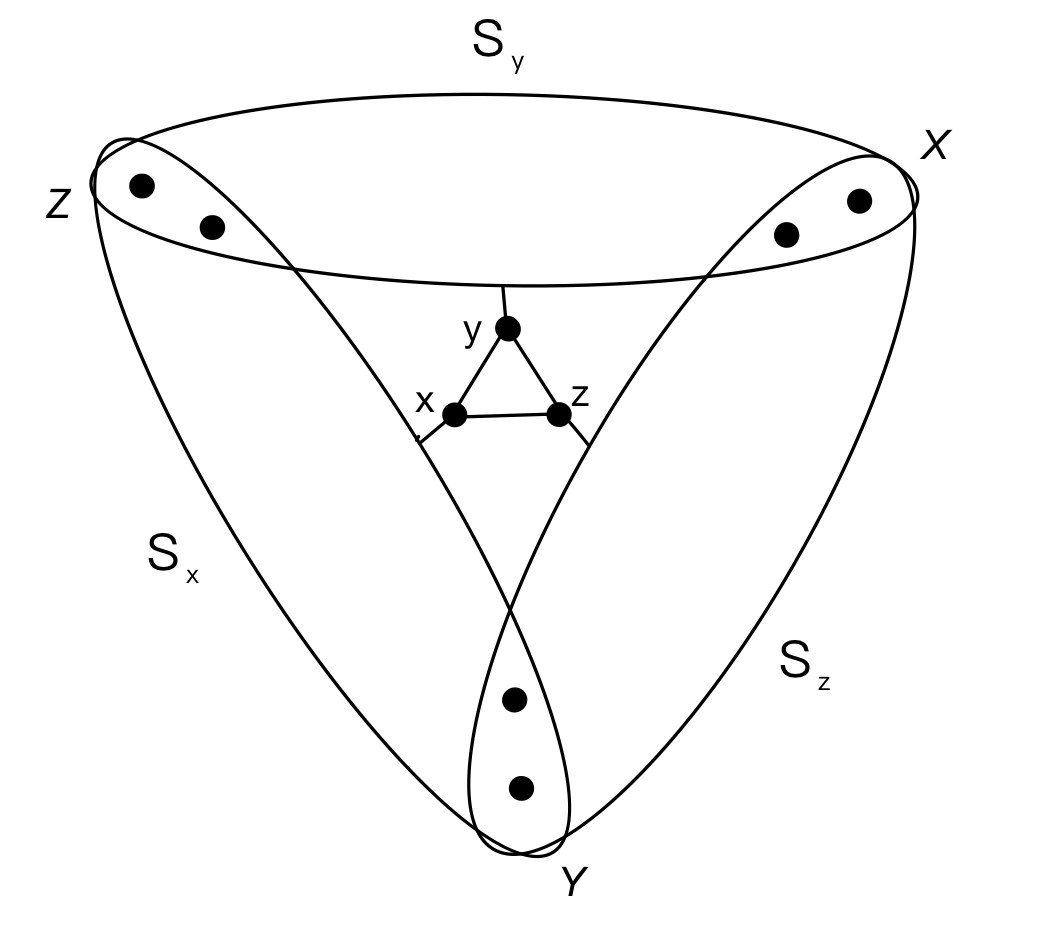}
\caption{The composition of $T$}
\label{three segments}
\end{figure}

Our next result concerns the additional edges on $S_x\cup S_y$, or equally, any other union 
of two segments above. Recall from Section \ref{segs} the concept of the core of a segment 
with respect to a handle.

\begin{lemma} \label{matching}
Let $C_x\subset S_x$ be the core of $S_x$ with respect to $Z=S_x\cap S_y$ and, 
symmetrically, let $C_y\subset S_y$ be the core of $S_y$ with respect to $Z$. Then every 
edge within $S_x\cup S_y$, that is not fully contained in $S_x$ or $S_y$, connects a vertex 
from $C_x$ with a vertex from $C_y$. Furthermore, such edges form a matching between $C_x$ 
and $C_y$.
\end{lemma}

\begin{proof}
Let $u\in S_x$. If $u\in Z$ then every edge from $u$ to a vertex of $S_y$ is contained in 
$S_y$. If $u$ is contained in the second handle, $Y$, of $S_x$ then $u$ and $y$ already 
have $\mu=2$ common neighbours, namely, $x$ and $z$, and so $u$ has no neighbours in $S_y$. 
Similarly, if $v\in S_x\setminus(Y\cup Z)$ but $v$ is adjacent to a vertex $z_i$ in $Z$ then 
the common neighbours of $u$ and $y$ are $x$ and $z_i$, and so again $u$ has no further 
neighbours in $S_y$. Hence only the vertices from $C_x$ can have further neighbours in 
$S_y$. Symmetrically, only vertices from $C_y$ can have further edges in $S_x$. We have 
shown that the edges on $S_x\cup S_y$ that are not fully in $S_x$ or $S_y$ connect $C_x$ 
with $C_y$. 

Now let $u\in C_x$. Then $x$ is a common neighbour of $v$ and $y$. Since $\mu=2$ and since 
$u$ is not adjacent to $Z$, it must have exactly one neighbour in $S_y$, and by the above, 
this unique neighbour is in $C_y$. Thus, every vertex in $C_x$ has a unique neighbour in 
$C_y$ and, symmetrically, every vertex of $C_y$ has a unique neighbour in $C_x$; that is, we 
have a matching between $C_x$ and $C_y$, as claimed.
\end{proof}

Needless to say, we have similar matchings between the corresponding cores in $S_x$ and 
$S_z$ and in $S_y$ and $S_z$.

\medskip
The process of forming $T=S_x\cup S_y\cup S_z$ consists of several steps. We first glue 
together possible segments $S_x$ and $S_y$ by identifying corresponding handles, which 
become $Z=S_x\cap S_y$. Then we select and add a matching between the cores of $S_x$ and 
$S_y$ with respect to $Z$. Clearly, the segments we identify must be of the same type (edges 
or non-edges) and then the two cores do indeed have the same size and matchings are possible. 
We note that matchings between two sets of size $m$ are in a natural correspondence with the 
elements of the symmetric group $S_m$, and so we have $|S_6|=720$ possible matchings if $Z$ 
is an edge, and $|S_4|=24$ possible matchings if $Z$ is a non-edge.

At this point, after choosing a particular matching, we already known the induced graph on 
$S_x\cup S_y$ and so we can write the Gram matrix corresponding to the set of unit vectors 
$S_x\cup S_y$ in $E$. This could, in principle, lead to elimination of some configurations 
due to negative eigenvalues of their Gram matrices. However, as we discovered 
computationally, all double unions survive this criterion, and this is why we are forced to 
consider a larger triple union of segments. Hence, the next step is to glue in a third 
possible segment $S_z$ by identifying its two handles with the remaining unmatched handles in 
$S_x$ and $S_y$. This is followed by a selection of matchings between the cores in $S_x$ and 
$S_z$ with respect to $Y=S_x\cap S_z$ and between the cores in $S_y$ and $S_z$ with respect 
to $X=S_y\cap S_z$. This completes forming of the induced graph on $T$ and so the whole Gram 
matrix on $T$ can be formed and checked for negative eigenvalues. Since $|T|=30$ is close to 
the dimension $34$ of the ambient Euclidean space $E$, this check is now quite powerful and 
eliminates a great majority of cases.

\medskip
There are, of course, a lot of technical details concerning this process, which we will 
provide later. In the remainder of this section we discuss gluing two graphs over a common 
subgraph.

\begin{definition}
A \emph{gluing} of graphs $A$ and $B$ over isomorphic subgraphs $H_A\subseteq A$ and 
$H_B\subseteq B$ is an isomorphism $\phi:H_A\to H_B$.
\end{definition}

A gluing produces a new graph by taking the union of $A$ and $B$ and, furthermore, 
identifying every vertex $h$ of $H_A$ with the corresponding vertex $\phi(h)$ of $H_B$. Hence 
$H_A$ and $H_B$ merge into the intersection $H=A\cap B$. Note that the edges of the glued 
graph all come from the edges of $A$ and $B$. Moreover, since $\phi$ is a graph isomorphism, 
edges of $A$ within $H_A$ merge with the corresponding edges of $B$ within $H_B$, so we do 
not end up with double edges.

Clearly, $H_A$ and $H_B$ must indeed be isomorphic or else no gluing is possible. We will 
however focus on a different question: how many different graphs can we obtain by gluing the 
given $A$ and $B$ over the given isomorphic $H_A$ and $H_B$? This is answered by the 
following proposition, adapted from statement (2.7) in \cite{Goldschmidt} to the 
graph-theoretic context. Instead, of gluing $H_A$ directly with $H_B$, we can introduce an 
independent copy $H$ of this graph and glue it onto $H_A$ and $H_B$ via all possible 
isomorphisms $\psi_A:H\to H_A$ and $\psi_B:H\to H_B$. This point of view stresses symmetry 
between $A$ and $B$ while being equivalent to our original gluing construction. In what 
follows we fix arbitrary isomorphisms $\gm_A:H\to H_A$ and $\gm_B:H\to H_B$.

Let $\Aut(A,H_A)$ be the group of all automorphisms of $A$ leaving $H_A$ invariant. 
Symmetrically, let $\Aut(B,H_B)$ be the group of all automorphisms of $B$ stabilising $H_B$. 
These automorphisms can be transferred into $\Aut(H)$ by conjugating with $\gm_A$ and 
$\gm_A$, respectively, giving us subgroups $\Aut(A,H_A)^{\gm_A}$ and $\Aut(B,H_B)^{\gm_B}$ 
of $\Aut(H)$. (Naturally, every element of $\Aut(A,H_A)$ is first restricted to $H_A$ and 
then conjugated by $\gm_A$, and similarly, for the elements of $\Aut(B,H_B)$.)

\begin{proposition} \label{gluing}
The number of non-isomorphic graphs formed by gluing $A$ and $B$ with respect to isomorphic 
subgraphs $H_A$ and $H_B$ coincides with the number of double cosets in the group $\Aut(H)$ 
of its subgroups $\Aut(A,H_A)^{\gm_A}$ and $\Aut(B,H_B)^{\gm_B}$.
\end{proposition}

Note that here we only consider isomorphisms between resulting glued graphs that preserve $A$ 
and $B$ set-wise. In principle, there could be further isomorphisms, for example, the ones 
switching $A$ and $B$ or even more general ones. These are not accounted for in the above 
proposition. However, this will be irrelevant for our purposes. The final remark is that the 
complete set of different gluings $\phi:H_A\to H_B$ is given by $\{\gm_B\al_i\gm_A^{-1}\mid 
i=1,2,\ldots k\}$, where $k$ is the number of double cosets above and $\al_1,\al_2,\ldots,
\al_k$ are representatives of the double cosets.

In our enumeration, the above proposition is applied to gluing the segments $S_x$ and $S_y$ 
over the respective handles. As we already mentioned, the handles must be of the same type: 
either both edges or both non-edges. Note that in either case $\Aut(H)$ is of order $2$ and 
the number of double cosets is one if the two vertices of the glued handle can be switched in 
the (handle-preserving) automorphism group of either of the two segments $S_x$ and $S_y$. If 
the two vertices cannot be switched then the number of double cosets is two. Clearly, this 
gives us an easy way to enumerate all \emph{segment pairs} (graphs obtained by gluing $S_x$ 
and $S_y$) up to isomorphism. 

In order to state the result of the enumeration, recall that we dropped all the segments of 
type $(4,6)$, and we are always gluing the first handles from both segments. Therefore, a 
segment pair can only be made of two segments of type $(6,6)$, or a segment of type $(6,6)$ 
and a segment of type $(6,4)$, or two segments of type $(6,4)$, or two segments of type 
$(4,4)$. 

The following statement, established by enumeration, gives us the number of different segment 
pairs.

\begin{proposition} \label{segment pair numbers}
There are $86333$ different segment pairs in total. Among them, there are $281$ segment pairs 
made of two segments of type $(6,6)$, $2249$ segment pairs made of a segment of type $(6,6)$ 
and a segment of type $(6,4)$, $4851$ segment pairs made of two segments of type $(6,4)$, and 
$78952$ segment pairs made of two segments of type $(4,4)$.
\end{proposition}

From this statement, we see that the last type of segment pairs is by far the most numerous. 
However, this is offset by the fact that the number $24$ of possible matchings between the 
cores of $S_x$ and $S_y$ arising in this case is much smaller than the number $720$ of 
matchings required for the first three types. Hence the four types of segment pairs are in 
fact reasonably balanced in computational terms.

\medskip
We next discuss the algorithmic details of all these processes.

\section{Steps 1 and 2: Building up $T$}\label{sec6}

We keep all possible segments in a list according to the type of the segment. First, we have 
the $19$ segments of type $(6,6)$, then the $78$ segments of type $(6,4)$, and finally, the 
$303$ segments of type $(4,4)$.

Each segment $S$ is stored with its vertices ordered in a particular way. First, we have the 
two handles, $X$ and $Y$, then the vertices from $S_0=S\setminus(X\cup Y)$ not contained in 
the cores $C_X$ and $C_Y$, then the vertices in $C_Y\setminus C_X$, then $C_X\setminus C_Y$, 
and finally the vertices from $C_X\cap C_Y$. This is shown in the following table, where the 
numbers in the third column indicate the size of the part. Note that the last four numbers 
constitute the quad type of $S$.
\begin{table}[H]
\centering
\begin{tabular}{|c|c|c|}
 \hline
 & $X$ & 2\\ \hline
 & $Y$ & 2\\ \hline
 none & $S\setminus(X\cup Y\cup C_X\cup C_Y)$ & $n_S$ \\ \hline
 right & $C_Y\setminus C_X$ & $r_S$ \\ \hline
 left & $C_X\setminus C_Y$ & $l_S$ \\ \hline
 both & $C_X\cap C_Y$ & $b_S$ \\ \hline
\end{tabular}
\caption{Vertices in a segment}
\label{quad type figure}
\end{table}

The first column indicates the meaning of the part: the two handles do not require explanation, 
the following part consists of vertices in neither of the two cores, hence the label `none'; 
the label `right' indicates that those vertices belong only in the core for the second handle; 
`left' means that these are in the core for the first handle; finally, `both' indicates that 
these vertices are in both cores. (`Left'/`right' was a useful mnemonic that we adopted during 
this project.)

When we enumerate segment pairs, we always assume that $S_x$ either precedes $S_y$ on the list 
or $S_x$ is the same segment as $S_y$. This guarantees that we do not overcount segment pairs.
As we have already mentioned, we glue the first handle of $S_x$ onto the first handle of 
$S_y$. 

\medskip
Once we select a specific segment pair from our list of $86333$, we need to go at Step 1 
through all possible matchings between the cores of $S_x$ and $S_y$ with respect to the handle 
$Z=S_x\cap S_y$, the first (`left') handle in both segments. These two cores are identified 
within the corresponding segment records as the union of the parts marked with `left' and 
`both', i.e., the cores are at the end of the record for both segments. Note that the two 
cores are of the same cardinality $c\in\{4,6\}$, as $Z$ has the same type in $S_x$ and $S_y$.

Recall that our method involves checking the Gram matrix on an ever growing set of vertices. 
In other words, once we know all edges within the subset, i.e., the Gram matrix on this 
subset is fully known, we immediately want to check this Gram matrix for semi-positive 
definiteness. In principle, such a check becomes more and more expensive as the subset 
becomes bigger. However, as we build the set up by adding one vertex at a time, this amounts 
to adding one new row (and column) to the Gram matrix. Fortunately, the LDLT algorithm 
\cite{NumRecipesAoSC}, which we use to verify absence of negative eigenvalues, is iterative 
exactly in this sense, and so we can significantly save on time by doing just the iterative 
part corresponding to adding one row. The details of this realisation of the LDLT algorithm 
are in Appendix \ref{appendix LDLT}.

Having the cores at the end of the segment record means that the induced subgraph on the 
first $22-c$ vertices of the segment pair (the full size of a segment pair is $12+12-2=22$ 
vertices) is the same regardless of the matching we add. That is, we need to account for the 
possible matchings only when dealing with the final $c$ vertices.

Recall that possible matchings between two cores $C_x\subset S_x$ and $C_y\subset S_y$ of 
equal size $c$ are indexed by the elements of the symmetric group $S_c$. Since $c\in\{4,6\}$, 
we have $6!=720$ matchings when $Z$ is an edge and $4!=24$ matchings when $Z$ is a non-edge. 
Looking at Proposition \ref{segment pair numbers}, we see that there are significantly more 
segment pairs where $Z$ is a non-edge, so adding the matchings evens out those two cases.

To add a matching to a given segment pair, we scan a pre-computed tree of depth $c$, which at 
each level $i$ decides the neighbour in $C_x$ of the $i$th vertex $y_i$ from $C_y$. This 
allows us to add the data of $y_i$ to the Gram matrix and check semi-positive definiteness. 
This results in a very efficient enumeration algorithm. Note that we had the option of using 
symmetry of the segment pair to cut down on the number of possible segment pairs with 
matching. However, in most cases the symmetry is trivial, and so we decided against using it, 
as it would not bring us much benefit and it would require a different enumeration tree for 
each symmetry type, significantly complicating the algorithm.

As it turns out, none of the resulting segment pairs with matching (corresponding to the 
leaves of the enumeration tree) are eliminated by the semi-positive definiteness criterion. 
This is why we next add the third segment $S_z$ at Step 2. Again, we could use the symmetries 
of the segment pair $S_x\cup S_y$ and the segment $S_z$ to reduce, using Proposition 
\ref{gluing}, the number of ways of merging the first and second handles in $S_z$ with second 
handles in $S_x$ and $S_y$, respectively. However, we again decided against it, because we 
judged that it would not give us a significant decrease in runtime to justify the trade off, 
a less transparent algorithm. Hence we simply allowed all four ways of gluing the two handles 
of $S_z$ onto the second handles of $S_x$ and $S_y$. In hindsight, this may have been a 
questionable decision as the majority of the hardest cases came from symmetric configurations 
and so, in fact, it may have significantly increased the overall runtime by doing some slow 
configurations twice. 

Note that we again assume that $S_z$ does not precede $S_y$ (and hence also $S_x$) in the 
list of segments, and this allows us to avoid the most obvious overcounting.

\medskip
Once the handles of $S_z$ are identified with the corresponding handles of $S_x$ and $S_y$, 
we still need to add two matchings, between the cores in $S_x$ and $S_z$ corresponding to the 
handle $Y=S_x\cap S_z$ and between the cores in $S_y$ and $S_z$ corresponding to the handle 
$X=S_y\cap S_z$. Just like we used an enumeration tree to account for all matchings in the 
segment pair $S_x\cup S_y$, we utilise a similar idea and account for the two new matchings 
within a single precomputed tree of depth $8$ (the cardinality of $S_z\setminus(X\cup Y)$). 
Hence, at each level $i$, the tree chooses the additional neighbours of the $i$th vertex 
$z_i$ of $S_z\setminus(X\cup Y)$ in $S_x$ and $S_y$. Referring to Table \ref{quad type 
figure}, if $z_i$ is in the `none' part of the segment $S_z$ then it has no additional 
neighbours and so we can immediately add this vertex to the Gram matrix and check 
semi-positive definiteness. If $z_i$ is in the `right' part of $S_z$ then it has an 
additional neighbour in $S_y$. Symmetrically, if $z_i$ is in the `left' part then it has an 
additional neighbour in $S_x$. Finally, if $z_i$ is in the `both' part of $S_z$, it has two 
additional neighbours, one in $S_x$ and the other in $S_y$. Again, once the neighbours of 
$z_i$ are selected, we add the data of $z_i$ to the Gram matrix and check semi-positive 
definiteness.

From this discussion, it is clear that the exact structure and size of this second 
enumeration tree (we called it the \emph{big tree}, as opposed to the smaller tree we use for 
adding the matching between the cores in $S_x$ and $S_y$) depends on the quad type $(n_S,r_S,
l_S,b_S)$ of $S=S_z$. We precomputed the big trees for all quad types and then simply use 
the correct one depending on the quad type of $S_z$. Note that while the structure and the 
total size of the big tree varies from one quad type to another, the number of final 
configurations (leaves) in the big tree depends only on the types of the handles $X$ and $Y$. 
Namely, we have exactly $6!6!=720^2=518400$ leaves when both $X$ and $Y$ are edges, 
$6!4!=720\cdot 24=17280$ leaves if $Y$ is an edge and $X$ is a non-edge, and $4!4!=24^2=576$ 
leaves if both $X$ and $Y$ are non-edges. 

These numbers indicate the possible numbers of complete configurations $T=S_x\cup S_y\cup 
S_z$ for each choice of a segment pair with a matching, $S_x\cup S_z$ with the third segment 
$S_z$ already glued in over the handles. This makes for really astronomical total numbers of 
possible configurations $T$, which would not be possible to enumerate and evaluate. What 
makes it a feasible project in the end is that we add vertices one by one and the 
semi-positive definiteness criterion kicks in a very non-trivial way after just a few 
(typically four or five) vertices of $S_z\setminus(X\cup Y)$ are added. So we rarely, in 
fact, almost never, have to visit in our enumeration all leaves of the big tree. 

\medskip
We do not have the exact data which proportion of all configurations is eliminated by 
expanding $S_x\cup S_y$ to the full set $T=S_x\cup S_y\cup S_z$. By our observation, it is 
well in excess of $99\%$. However, a tiny proportion of survivors ends up being a significant 
number of full configurations $T$ which are semi-positive definite, and hence, in order to 
try and eliminate those, we need an extra step in the algorithm, adding vertices beyond $T$.

\section{Step 3: Beyond $T$}

After extensive experiments trying to find a meaningful way to add additional vertices, we 
settled on the following scheme: we add vertices from $G\setminus\hat T$, where $\hat T=Q\cup 
T$, that are adjacent to the first vertex $t$ from the handle $X=S_y\cap S_z$. This is the 
vertex number $13$ in $T$.

\subsection{Enumerating additional vertices}

Note that $t$ has no neighbours in $S_x$ and it has two neighbours in both $S_y$ and $S_z$. 
Hence, within $\hat T$, the vertex $t$ has $2+2+2=6$ neighbours, if $X$ is a non-edge, and $t$ 
has $2+1+1+1=5$ neighbours in $\hat T$, if $X$ is an edge\footnote{For the reasons that will 
be explained later, we never encounter in the actual calculation the case where $X$ is an 
edge.}. In any case, $t$ has at least $14-6=8$ neighbours in $G\setminus\hat T$, and these are 
the vertices we aim to add to our configuration. Adding eight extra vertices to $T$ brings the 
total to $38$ vertices, which significantly exceeds the dimension $34$ of $E$. Hence not only 
the Gram matrix on this larger set must be semi-positive definite, but also the rank of the 
Gram matrix cannot exceed $34$, which means that the radical must be of dimension at least 
$38-34=4$, and this is a super-strong condition. On the other hand, we do not know adjacency on 
the set of extra vertices, as we do not utilise any concept similar to segment here. We add the 
extra vertices one by one, using the restrictions that we now proceed to discuss.

\begin{lemma} \label{neighbours in T}
Suppose that $u\in G\setminus\hat T$, and $u$ is adjacent to $t$. Then $u$ has exactly two 
neighbours in $S_x$  and it has exactly one additional (other than $t$) neighbour in both $S_y$ 
and $S_z$.
\end{lemma}

This is immediate since $u$ is not adjacent to any vertex in $Q$ and $\mu=2$. This means that 
$u$ has the maximum of five and the minimum of three (if $u$ is adjacent to vertices in both 
handles $Y$ and $Z$) neighbours in $T$.

We pre-compute all such possible sets of neighbours in $T$.

\begin{lemma}
In total, there are exactly $2080$ possible configurations of neighbours of $u$ in $T$ that 
satisfy Lemma $\ref{neighbours in T}$.
\end{lemma}

Note that this set of possible sets of neighbours of extra vertices $u$ only depends on how 
the segments are embedded in $T$, and so this calculation needs to be done only once, as its 
result, the array we call \verb|downs|, is applicable to all $T$. However, a concrete $T$ 
allows further elimination of some of these possibilities.

First of all, we note that each $u$ is uniquely identified by its neighbours in $T$. 

\begin{lemma}
If $u'\in G\setminus\hat T$ has the same neighbours in $T$ as $u$ then $u'=u$.
\end{lemma}

\begin{proof}
Indeed, suppose that $u'\neq u$. Since they have at least three common neighbours in $T$ and 
$\mu=2$, we must have that $u$ and $u'$ are adjacent. Furthermore, since $\lm=3$, $u$ and $u'$ 
have exactly three neighbours in $T$, one in each handle. Let $t'$ be their common neighbour 
in the handle $Z=S_x\cap S_y$. Then we know from Lemma \ref{segnonadj} that $t$ and $t'$ are 
non-adjacent and, at the same time, they have three common neighbours, $y$, $u$ and $u'$. This 
is a contradiction proving the claim.
\end{proof}

Hence, each element $d$ of \verb|downs| describes a unique potential neighbour $u$ of $t$.

\subsection{Eliminating impossible $d$}

To eliminate some $d$, we first compute the demand for all pairs of vertices from $T$. That 
is, given a pair $t_i,t_j\in T$, with $i<j$, we compute how many vertices from $G\setminus\hat 
T$ should there be that are adjacent to both $t_i$ and $t_j$. Namely, we start with the total 
of $\lm=3$, if $t_i$ and $t_j$ are adjacent, and the total of $\mu=2$, if they are 
non-adjacent. Then we subtract one from this total for each common neighbour of $t_i$ and 
$t_j$ in $\hat T$. If we end up with a negative demand for some pair $t_i,t_j$ then we can, 
clearly, discard this particular configuration $T$ altogether. The diagonal entries 
in the demand matrix, corresponding to the situation $t_j=t_i$, are not used and hence 
are not computed.

Once the demand matrix is known and it does not contain negative values, we check every $d$ 
in \verb|downs| against it. If for a pair $i,j\in d$, $i<j$, we have that the demand for 
$t_i$ and $t_j$ is zero then we discard such a $d$.

If $d$ survives this check, we then check it for the following condition: for each $i\in
\{1,2,\ldots,30\}\setminus d$, we check that the number of known common neighbours between 
$t_i$ and $u$ (corresponding to $d$) does not exceed $\mu=2$. (Note that since $i\notin d$, 
$t_i$ and $u$ are not adjacent.) If such an $i$ is found then $d$ is discarded as well. We 
also compute at this stage what we call the \emph{halo} of the potential vertex $u$. By 
definition, this is the set of all $i\in\{1,2,\ldots,30\}\setminus d$ such that $t_i$ and $u$ 
have exactly $\mu=2$ known common neighbours. The halo is used at a later stage when we check 
for possible adjacency among additional vertices $u$. Note that to find the number of known 
common neighbours for $t_i$ and $u$ we simply count the number of $j\in d$ such that $M_{ij}
=\frac{2}{7}$, where $M$ is the Gram matrix of $T$, which is known at that point.

The next check for $d$ involves verification that adding $u$ to $T$ does not create a non 
semi-positive Gram matrix. First, we compute the vector $r=(r_1,r_2,\ldots,r_{30})$ of all 
values $u\cdot t_i$. Hence 
$$
r_i=
\left\{
\begin{array}{rl}
\frac{2}{7},&\mbox{if }i\in d;\\
-\frac{1}{14},&\mbox{otherwise}.
\end{array}
\right.
$$
This vector allows us to find $u\cdot w$ for every $w=\sum_{i=1}^{30}c_it_i\in W=\la T\ra$ 
by simply computing the $1\times 1$-matrix $rc^T$, where $c=(c_1,c_2,\ldots,c_{30})$.

Non semi-positive definiteness of the extended Gram matrix can manifest itself in two 
different ways. First, it arises when $rc^T\neq 0$ for some vectors $c$ such that 
$u=\sum_{i=1}^{30}c_it_i=0$. We can verify this condition as follows: while checking $T$ for 
semi-positive definiteness via the LDLT decomposition, we also compute the matrix $R=L^{-1}$, 
with the rows $R_i$ of $R$ representing an orthogonal basis in $\RR^{30}$ endowed with the 
symmetric bilinear form represented by the Gram matrix $M$. (See Appendix \ref{appendix LDLT} 
for details.) In particular, the vectors $R_i$ such that $D_{ii}:=R_iMR_i^T=0$ form a basis 
in the null space of $M$, and so we can simply check the above condition for all such 
$c=R_i$.

The second way the extended Gram matrix may become non semi-positive definite is when the 
projection $\proj_W(u)$ of $u$ to $W$ is longer that $u$, and so $u-\proj_W(u)$ has negative 
(square) length. For this check, we compute the projection matrix $P$ for $T$. (See Appendix 
\ref{projection} for the details.) Then the projection of $u$ is found as $\sum_{i=1}^{30}p_i
t_i$, where $p=rP$. Finally, we find the length of the projection as $rp^T$, and if the entry 
in this $1\times 1$-matrix exceeds $1$ (the length of $u$) then we reject $d$.

Filtering out all impossible elements of \verb|downs| via the above conditions leaves us with 
the array \verb|verts| of all $d$ that may potentially correspond to vertices in $G\setminus
\hat T$. Turning to the computational aspect of this, in most cases \verb|verts| would only 
be a small part of \verb|downs|, having no more than $150$ possible $d$. (Recall that 
\verb|downs| has size $2080$.) However, in some difficult cases we observed larger arrays 
\verb|verts| of up to $600+$ sets $d$. The checks themselves were quite quick for each $T$, 
but the later calculations could be long when \verb|verts| was large.

\subsection{Compatibility of additional vertices}

Now that the set of possible additional vertices, represented by the array \verb|verts|, has 
a more reasonable size, we can try to select from it the eight neighbours of $t$. We do the 
selection recursively. The current (incomplete) selection of vertices is represented by the
array \verb|further| and at every stage we have the current unsatisfied demand matrix and the 
current version of \verb|verts|, where the additional vertices $u$ have been checked for 
compatibility with the already selected vertices in \verb|further|. Hence we now proceed to 
describe these compatibility conditions.

If $u$ and $u'$ are two such vertices, corresponding to $d$ and $d'$, then we have that they 
must eventually form an edge or a non-edge. These two options involve different checks. 
For a potential edge, we verify that:
\begin{enumerate}
\item[(a)] the intersection $d\cap d'$ is of size at most $3$; this is because $\lm=3$;
\item[(b)] $d$ meets trivially the halo of $u'$ and, symmetrically, $d'$ meets trivially 
with the halo of $u$; and 
\item[(c)] adding both $u$ and $u'$ to $T$ does not lead to non semi-positive definiteness, 
assuming that $uu'$ is an edge.
\end{enumerate}

We need to justify (b) and explain how the check (c) is performed.

\begin{lemma}
For vertices $u,u'\in G\setminus\hat T$ with sets of neighbours in $T$ given by $d$ and $d'$, 
respectively, if $d$ meets the halo of $u'$ non-trivially then $u$ and $u'$ cannot be 
adjacent in $G$.
\end{lemma}

\begin{proof}
Recall that the halo of $u'$ consists of all vertices $t_i$, with $i\notin d'$, such that 
$t_i$ has exactly two neighbours, $t_j$ and $t_k$, with $j,k\in d'$. If $t_i$ is adjacent to 
$u$ and $u$ is adjacent to $u'$ then $u$, $t_j$, and $t_k$ are common neighbours of $t_i$ and 
$u'$, which is a contradiction, because $t_i$ and $u'$ are non-adjacent and $\mu=2$.
\end{proof}

This shows that indeed (b) is a valid check.

As for the check in (c), we use the vectors $r$ and $r'$ corresponding to $u$ and $u'$ (see 
above) and also the corresponding vectors $p=rP$ and $p'=r'P$, which are all known at this 
point. Note that, setting $w=\proj_W(u)$, $v=w-u$, $w'=\proj_W(u')$, and $v'=w'-u'$, we 
must have that $v\cdot v'+w\cdot w'=u\cdot u'=\frac{2}{7}$, if $u$ and $u'$ are adjacent. 
Consequently, we must have that $|\frac{2}{7}-w\cdot w'|=|v\cdot v'|\leq|v||v'|$, and hence 
$(\frac{2}{7}-w\cdot w')^2\leq(v\cdot v)(w'\cdot w')=(u\cdot u-w\cdot w)(u'\cdot u'-w'\cdot 
w')=(1-w\cdot w)(1-w'\cdot w')$. This results in the condition we verify in (c):
$$
\left(\tfrac{2}{7}-r(p')^T\right)^2\leq\left(1-rp^T\right)\left(1-r'(p')^T\right).
$$
As usual, we identify the $1\times 1$ matrices here, such as, say, $r(p')^T$, with the entry 
in it. Also note that $w\cdot w'=u\cdot w'$ and so $r(p')^T$ correctly represents this value 
(and similarly for the other terms in the inequality).

The final remark about checking whether $u$ and $u'$ can form an edge is the the halo 
condition (b) is quite strong and it eliminates a lot of pairs $u$ and $u'$.

\medskip
To check whether $u$ and $u'$ can form a non-edge, we verify the following:
\begin{enumerate}
\item[(a)] the intersection $d\cap d'$ is of size at most $2$, since $\mu=2$; and
\item[(b)] adding both $u$ and $u'$ to $T$ does not lead to non semi-positive definiteness, 
assuming that $uu'$ is a non-edge.
\end{enumerate}

We do not have, unfortunately, a halo-type condition for non-edges. The condition (b) is 
verified via the inequality
$$
\left(-\tfrac{1}{14}-r(p')^T\right)^2\leq\left(1-rp^T\right)\left(1-r'(p')^T\right),
$$
which is similar to the condition (c) we had for edges.

\medskip
Finally, additional vertices $u$ and $u'$ are \emph{compartible} if they can form an edge, or 
a non-edge, or both. In most cases, both $u$ and $u'$ have very small length, which either 
eliminates the pair altogether, as non-compatible, or only one condition, for an edge or for 
a non-edge, could be satisfied. However, in difficult cases we may have many compatible 
pairs, about which we cannot decide immediately whether they form an edge or a non-edge.

\subsection{Recursion}

We have already mentioned that we build the possible sets \verb|further| of additional 
neighbours of $t$ recursively. The recursor function takes as arguments the current set of 
possible extra vertices \verb|verts| and current demand array. Recall that the demand array 
records for each pair $i,j\in\{1,2,\ldots,30\}$, $i<j$, how many additional common vertices 
of $t_i$ and $t_j$ we can still add. The current array \verb|further| is a global variable and 
it is affected (extended) by the recursor. Above we explained how to find the initial demand 
array and the initial $\verb|verts|$. Clearly, the initial \verb|further| is empty.

The recursor first checks if we have \emph{forced} vertices. These are the vertices that must 
be added if we are to satisfy the demand. For this, we compute the offer array counting the 
common neighbours of all pairs $t_i$, $t_j$ among the vertices in \verb|verts|.

If for some pair $i,j$, with $i<j$, demand exceeds offer then, clearly, demand cannot be met 
and so we exit the recursor right away, as this configuration cannot be successfully completed. 
If demand for $i,j$ is equal to offer then the only way to satisfy demand is by adding to 
\verb|further| all common neighbours of $t_i$ and $t_j$ from \verb|verts|, so these common 
neighbours are forced and we add them to the array \verb|forced|, and we do this procedure for 
all pairs $i,j$.

If this results in a non-empty array \verb|forced| then we do the following checks on it:
\begin{enumerate}
\item[(a)] the total length of \verb|further| and \verb|forced| does not exceed the number 
of additional neighbours $t$ can have (which is $8$ in the actual calculation); and
\item[(b)] the vertices in \verb|forced| are compatible with each other.
\end{enumerate}
(Note that that these vertices are compatible with all vertices from \verb|further|, as 
this holds for all vertices in the current \verb|verts|.) If either of the two checks above 
fails then the current configuration cannot be completed and we exit the recursor. 

If the set of forced vertices passes the checks then we add them to \verb|further| and 
compute the new demand matrix and the new \verb|verts| by removing from it the vertices 
that are 
\begin{enumerate}
\item[(a)] forced; or
\item[(b)] are adjacent to a pair $t_i$, $t_j$ with new demand zero; or
\item[(c)] are not compatible with some forced vertex.
\end{enumerate}
While computing the new demand, we check that it remains non-negative for all pairs $i,j$, or 
else we exit the recursor. Also, if the length of new \verb|further| is equal to $8$, the number 
of required additional neighbours of $t$, then we have arrived at one of the possible exact sets 
of additional vertices of $t$ and so we call a different function (that we describe in the next 
section) to see if this can actually lead to a graph $G$. Once the new demand and new 
\verb|verts| are computed, we call a new instance of the recursor. On return from it, we exit 
the current recursor, as nothing else can be done.

If, on the other hand, we find no forced vertices in \verb|verts| then we select $t_i\in 
T\setminus\{t\}$ so that the demand for the pair $t$ and $t_i$ is non-zero but as small as 
possible. The idea is that we must add an extra vertex satisfying this demand and the minimality 
condition hopefully means that we have a short list of possible additional vertices that we can 
use. Hence we make the list of vertices from \verb|verts| that are joint neighbours of $t$ and 
$t_i$, and we add to \verb|further| one vertex from this list in a loop. Note that when we add 
the $i$th vertex from this list, it means we have already tried all preceding vertices and so 
they can be removed from further consideration. We compute the new demand array and new 
\verb|verts|, as above. If \verb|further| has length exactly $8$, we again call the function 
deciding whether the exact set \verb|further| can lead to a graph $G$. Otherwise, we call a new 
instance of recursor. On exit from the call, we restore \verb|further| and continue with the 
loop, and when it ends we exit the current recursor, as there is nothing left for us to try.

\medskip
To summarise this section, if we have a complete set $T$ with a semi-positive definite Gram 
matrix then we recursively enumerate all possible sets of $8$ additional neighbours of the 
vertex $t=t_{13}$ from the handle $X=S_y\cap S_z$. In most cases this procedure is quite 
efficient and in a great majority of cases it does not produce any possible exact sets of 
additional neighbours of $t$, thus ruling $T$ out. However, in difficult cases, it can produce 
some exact candidates for the additional neighbours of $t$. Note that we cannot immediately check 
the extended set of vertices for semi-positive definiteness, as we may not know the edges among 
the additional vertices. So there is a further enumeration step to be done, and it is described 
in the next section.

\section{Step 4: Exact sets} \label{exact}

Suppose that we have a set $T$ with a semi-positive definite Gram matrix $M$ and a set 
\verb|further| of all additional (i.e., not contained in $T$) neighbours of $t=t_{13}$. Clearly, 
the graph $C=G_1(t)$ on the set of neighbours of $t$ should be a good cubic graph. The issue is 
that we do not know all edges in this local graph, but we do have some partial information about 
edges. First of all, we organise the vertex set $\{c_1,c_2,\ldots,c_{14}\}$ of $C$ as follows: 
we take $c_1=y$, $c_2=z$, $c_3$ and $c_4$ are the neighbours of $t$ in $S_y$, and $c_5$ and 
$c_6$ are the neighbours of $t$ in $S_z$. (Recall that the handle $X$ is a non-edge, and so our 
counting is correct.)  The remaining eight vertices $c_7,c_8,\ldots,c_{14}$ come from the array 
\verb|further|. 

We know that $y=c_1$ is adjacent to $z=c_2$, as well as $c_3$ and $c_4$, but not to any other 
vertex of $C$. Similarly, $z=c_2$ is also adjacent to $c_5$ and $c_6$, but not to any further 
vertex from $C$. Adjacency among the vertices $\{c_3,c_4,c_5,c_6\}$, which are in $T$, can be 
gleaned from the available Gram matrix $M$ of $T$, and so we know all edges there. Also known 
are all edges between the vertices $v\in\{c_3,c_4,c_5,c_6\}$ and $u\in\{c_7,c_8,\ldots,c_{14}\}$, 
because these are recorded in the element $d$ of \verb|downs| corresponding to $u$. However, for 
pairs of vertices $u,u'\in\{c_7,c_8,\ldots,c_{14}\}$, i.e., in \verb|further|, we only have 
partial information: they are compatible, which means that at least one of the two 
possibilities, an edge or a non-edge, has not been ruled out for each such pair. So the status 
of each pair $u,u'$ from \verb|further| is one of the following:
\begin{enumerate}
\item[(a)] definitely an edge;
\item[(b)] definitely a non-edge; or
\item[(c)] an edge or a non-edge.
\end{enumerate}

At this final stage, Step 4, of our enumeration algorithm we recursively go through all 
possibilities for the local graph $C$.

The preparation step for this recursion involves computing a $14\times 14$ matrix representing 
the current information about the edges of $C$, as above. Namely, for each pair $i,j$, the edge 
matrix records the current status of the pair $i,j$, according to the cases (a)-(c). We also 
compute the demand array, which, for each $i$, records how many edges the vertex $c_i$ is missing, 
and the supply array, which similarly records, for each $i$, how many $j\neq i$ are there such 
that the pair $i,j$ is recorded in the edge matrix as being case (c), i.e., $c_i$ and $c_j$ do 
not currently form an edge, but they may form an edge eventually. Note that if there is an $i$ 
such that the demand for $c_i$ is greater than supply then this clearly is an impossible situation 
and so we exit.

These three arrays, the edge matrix, demand list, and supply list serve as arguments of the Step 
4 recursor function. In this function, we first try to remove some uncertainties from the edge 
matrix, i.e., we try to transform each uncertain case (c) into one of the definitive (a) and (b). 
We can do this when one of the following conditions is met:
\begin{enumerate}
\item if the demand for some $i$ is zero (i.e., $c_i$ already has its three neighbours in $C$), 
but the supply for $i$ is not zero, we mark all undecided pairs $i,j$ as non-edges;
\item if the demand for some $c_i$ is equal to the supply then all the uncertain pairs $i,j$ are 
changed to edges;
\item for each pair $i,j$, we compute the current number $e$ of known common neighbours of $n_i$ 
and $n_j$ in $N$; then
\begin{enumerate}
\item if $e\geq 3$ then this is an impossible configuration, so we quit;
\item if $e=2$ then 
\begin{enumerate}
\item if $i,j$ is a known non-edge then this is a contradiction, since, in a good graph, two 
non-adjacent vertices can have at most one common neighbour; hence we quit;
\item if $i,j$ is currently recorded as uncertain then we change it to an edge for the same 
reason as above;
\item if $i,j$ is already known to be an edge then we make certain that $c_i$ and $c_j$ have no 
further common neighbours; namely, we go through all vertices $c_k\in C$, $k\neq i,j$, and if 
$i,k$ is an edge and $j,k$ is uncertain, we make it a non-edge; similarly, if $i,k$ is uncertain 
and $j,k$ is an edge, we make $i,k$ a non-edge;
\end{enumerate}
\item if $e=1$ then we can only force change if $i,j$ is a known non-edge; then we make sure that 
$c_i$ and $c_j$ have non further common neighbours, as above: for $k\notin\{i,j\}$, if $i,k$ is 
an edge and $j,k$ is uncertain, we make it a non-edge; similarly, if $i,k$ is uncertain and $j,k$ 
is an edge, we make $i,k$ a non-edge.
\end{enumerate}
\end{enumerate}
As we implement changes, we also update the demand and supply lists accordingly, and if supply 
is ever less than demand then we quit, as this is an impossible situation. Note that if one of 
the above checks yields a change in the edge matrix then this may have consequences for other 
vertices and hence we iterate the above checks until no further changes arise.

Now suppose we have removed as much uncertainty from the edge matrix as we could, but we can 
still find a pair $i,j$ that is uncertain. Then we try both possibilities for this pair:
\begin{enumerate}
\item we make $i,j$ an edge, update demand and supply accordingly, and call a new instance of the 
recursor;
\item on return, we make $i,j$ a non-edge, update supply and demand lists, and again call a new 
instance of the recursor;
\item on return from this second attempt, we quit, as there is nothing else we can do.
\end{enumerate}

Finally, if we managed to remove all uncertainty then $C$ is now a good graph and we know all 
edges within the set $T\cup C$. This allows us to find the Gram matrix $N$ for the projection of 
the vertex set of $C$ into the orthogonal complement $W^\perp$ of $W=\la T\ra$, which we can then 
check for semi-positive definiteness and rank. We compute this Gram matrix $N$ as follows. First 
of all, the projection of the vertex $c_i$ to $W^\perp$ is the vector $v_i=n_i-w_i$, where 
$w_i=\proj_W(c_i)$. Note that the six vertices in $C\cap\hat T$ are contained in $W$ and so they 
have zero projection to $W^\perp$. Therefore, we only need to take the remaining eight vertices 
$c_7,c_8,\ldots c_{14}$ from \verb|further|, and so $N$ is of size $8\times 8$. The entry of $N$ 
corresponding to the pair $i,j$ equals to $v_i\cdot v_j=c_i\cdot c_j-w_i\cdot w_i=c_i\cdot c_j
-w_i\cdot c_j$. Recall that the first term (minuend) here is equal to $1$, or $\tfrac{2}{7}$, or 
$-\tfrac{1}{14}$ when, correspondingly, $c_i$ and $c_j$ coincide, or they form and edge in $C$, 
or they form a non-edge. The second term (subtrahend) can be computed in the matrix form as 
$r_ip_j^T$, where the vectors $r_i$ and $p_j=r_jP$ are known, since we used them to check 
compatibility of $c_i$ and $c_j$. Thus, we have all the necessary ingredients for this calculation 
and can readily compute $N$.

Note that we need to do this calculation exceedingly rarely, so we approximate semi-positive 
definiteness in a crude way by simply checking that the determinants of the principal minors of 
$N$ are non-negative. We also compute the rank of $N$ and check that the sum of the ranks of $M$ 
(Gram matrix of $T$) and $N$ does not exceed the total embedding dimension of $34$. Since 
$|T\cup C|=38$, this latter condition is super strong and in fact it takes no survivors, and so 
all configurations are eliminated at this stage, completing the calculation. 

\section{Choosing $Q$}

We have described above a multi-stage process of eliminating possible $G=\srg(85,14,3,2)$ by 
enumerating and checking triple unions $T=S_x\cup S_y\cup S_z$ or, in harder cases, $T\cup C$, 
where $C=G_1(t)$ and $t=t_{13}\in T$. Note that we start it all with a $3$-clique $Q=\{x,y,z\}$, 
whose existence is guaranteed. In this short section we describe an improvement to our algorithm, 
whereby we select $Q$ in a controlled way and this results in disappearing of a significant number 
of segment triples we need to consider. 

Recall that $Q$ can be chosen for any $x\in G$ by selecting an edge $yz$ in $G_1(x)$ that is not 
contained within the good graph $G_1(x)$ in a larger clique. For each of our $39$ types of good 
graphs, $H$, we pre-select a \emph{favourite} edge $yz$ in $H$, not contained in a larger clique. 
We will call the segment $S=H\setminus yz$ the \emph{favourite segment} for the good graph $H$.

\begin{proposition} \label{choose Q}
The maximal $3$-clique $Q=\{x,y,z\}$ in $G$ can be chosen so that at least one of the following 
holds:
\begin{enumerate}
\item[(a)] $yz$ is the favourite edge in $G_1(x)$; or
\item[(b)] $xz$ is the favourite edge in $G_1(y)$; or
\item[(c)] $xy$ is the favourite edge in $G_1(z)$.
\end{enumerate}
That is, in the triple of segments $\{S_x,S_y,S_z\}$ at least one segment is favourite for its 
good graph.
\end{proposition}

The proof is immediate, and in fact, we could have claimed just one of the three options. However, 
the above symmetric form is needed because we only consider ordered triples of segments $S_x$, 
$S_y$, and $S_z$. Namely, we assume that $S_y$ does not precede $S_x$ in the list of segments and, 
similarly, $S_z$ does not precede $S_y$. If we simply select $yz$ to be favourite in $G_1(x)$ then 
we cannot be sure that $S_y$ and $S_z$ do not precede $S_x$. However, we can, clearly, change the 
order of vertices in $Q$, so that the order of three segments in $T$ is the correct one, and 
manifestly, the symmetric condition from Proposition \ref{choose Q} is then maintained.

\medskip
How do we select the favourite edge $yz$ in each good graph $H$? Our preference is for an edge such 
that both handles in the segment $S=H\setminus yz$ are non-edges (segment type $(4,4)$). In a small 
number of good graphs, such a choice is impossible, in which case we select the favourite edge $yz$ 
so that the first handle in $S$ is an edge and the second handle is a non-edge (segment type 
$(6,4)$) and this is always possible. So we never need to select a segment type $(6,6)$ as our 
favourite.

The consequence of such a choice is that at least one handle in $T$ is guaranteed to be a non-edge, 
and then, because of our ordering of segments, where the segments of type $(6,6)$ precede the 
segments of type $(6,4)$, which in turn precede the segments of type $(4,4)$, we can be sure that 
the handle $X=S_y\cap S_z$ is definitely a non-edge. The advantage of this is that we never 
encounter the largest possible big enumeration trees and also this guarantees that our count of 
$8$ additional vertices for $t\in X$ is correct.

Our approach with favourite edges also eliminates all triples of segments, where none of the 
segments is favourite. Overall, this improvement to the enumeration algorithm shaves off, by our 
estimate, close to two orders of magnitude from the total run time, and hence it contributes 
significantly to making the enumeration feasible.

\section{Conclusion}

We do not include in this paper the full enumeration code we produced, as it is quite long. It can 
be found on GitHub \cite{code}. We ran it in GAP on $96$ cores in four servers in the School of 
Mathematics at the University of Birmingham continuously for over a year from November 2023 to 
January 2025. Individual cases of segment pairs took anywhere from several second to several months 
on a single core. The longest $21$ cases had to be split up further between many cores so they could 
be completed. This final calculation used a slight modification of the same code. The longest of the 
$21$ cases took about a month on $32$ cores.

As we hopefully already made clear, none of the configurations $T$ or $T\cup C$ survived the 
complete checks, and this means that $\srg(85,14,3,2)$ does not exist.

\medskip
Can this method be generalised and used to study other unresolved case of strongly regular graphs? 
This remains to be seen. For it to be successful, we need to have a relatively low embedding 
dimension and at the same time a substantial part of the graph needs to be tight enough so it could 
be enumerated within a reasonable amount of time. One possible candidate is Conway's 
$\srg(99,14,1,2)$\footnote{Apparently, John Conway was interested in this set of parameters and he 
even offered a monetary reward to anyone who could enumerate this case of strongly regular graphs.}.

\appendix

\section{Enumeration trees}

In this paper, we provide a rather detailed description of the enumeration code we used. This is 
because the efficient realisation of the huge enumeration is what makes this entire project 
feasible. In this appendix we describe the enumeration trees used in the algorithm. Recall that it 
consists of four steps. and the enumeration trees are used for the first two. At Step 1, we use 
an enumeration tree to go through all possible ways to attach a matching to a pair of segments 
$S_x\cup S_y$ joined at the handle $Z=S_x\cap S_y$. The matching we need to add is between the cores 
$C_x$ in $S_x$ and $C_y$ in $S_y$ corresponding to the handle $Z$.

We have two cases: (a) if $Z$ is an edge then $|C_x|=|C_y|=6$; and if $Z$ is a non-edge 
then $|C_x|=|C_y|=4$. Correspondingly, we have two trees and we refer to them as the 
\emph{small trees}.

Let $n\in\{4,6\}$ be the size of the cores. The record at the node of a small tree has 
the following structure.
$$
\renewcommand\arraystretch{1.5}
\begin{array}{|r|l|}
\hline
\mbox{Level}&k\\
\hline
\mbox{Neighbour}&m\\
\hline
\mbox{Brother}&bro\\
\hline
\mbox{Son}&son\\
\hline
\end{array}
$$
The level $k\in\{1,2,\ldots,n\}$ here refers here both to the depth of this node in the tree and the 
number of the vertex in $C_y$ for which we now need to choose a neighbour in $C_x$. When we operate 
with this node record, the neighbours of the first $k-1$ vertices of $C_y$ have already been chosen 
at the lower levels of the tree and so the current neighbour $m$ should be different from all of 
those earlier neighbours.

The brother and son entries specify the tree structure: the brother entry refers to the next node 
with the same parent (and hence also at the same level). If the current node is the last one for its 
parent then the brother entry is set to zero. The son entry refers to the first descendent node at 
the level $k+1$ and it is set to zero if the current node is at the deepest level $k=n$. 

Clearly, the leaves of this tree correspond to complete matchings and so there are exactly $n!\in
\{24,720\}$ of them. Using this tree structure, we have a simple code that allows us to enumerate 
all matchings while iteratively computing the necessary data, such as the LDLT decomposition of the 
partial Gram matrix. (See Appendix \ref{appendix LDLT}.)

\medskip
We also use an enumeration tree at Step 2, where we have a segment pair $S_x\cup S_y$, complete with 
a matching as above, and the third segment $S_z$, already glued to $S_x\cup S_y$ via the handles 
$Y=S_x\cap S_z$ and $X=S_y\cap S_z$. At this point we need to add matchings between the cores in 
$S_z$ and $S_x$ corresponding to the handle $Y$ and between the cores in $S_z$ and $S_y$ 
corresponding to the handle $X$. 

The structure of the \emph{big tree}, which allows us to to enumerate both matchings at the same 
time, depends on the quad type of the segment $S=S_z$ (see Table \ref{quad type figure} and the 
discussion there). We note that the handles $Y$ and $X$ are second handles in $S_x$ and $S_y$ and 
they are the first and second handles in $S_y$, which is opposite to the order in Table \ref{quad 
type figure}.

The node record of a big tree is as follows.
$$
\renewcommand\arraystretch{1.5}
\begin{array}{|r|l|}
\hline
\mbox{Level}&k\\
\hline
\mbox{Left}&l\\
\hline
\mbox{Right}&r\\
\hline
\mbox{Brother}&bro\\
\hline
\mbox{Son}&son\\
\hline
\end{array}
$$
Here $k$ refers to the depth of the node in the tree and at the same time it refers to the number 
of the vertex in the segment $S\setminus(X\cup Y)$, where we disregard the two handles because they 
are contained in $S_x\cup S_y$. Hence the total depth of a big tree is $12-4=8$, regardless of quad 
type. Recall from Table \ref{quad type figure} that $S\setminus(X\cup Y)$ consists of four groups. 
Depending on the group, a vertex can have no new neighbours (`none'), only a neighbour in $S_y$ 
(`right'), only a neighbour in $S_x$ (`left'), or finally a neighbour in $S_x$ and a neighbour in 
$S_y$ (`both'). Clearly, here `left' refers to $S_x$ and `right' refers to $S_y$. The union of the 
`right' and `both' groups is the core of $S$ with respect to the second handle of $S$ (currently, it 
is $X=S_y\cap S_z$) and, similarly, the union of the `left' and `both' groups constitute the core of 
$S$ with respect to its first handle (currently, $Y=S_x\cap S_z$). Hence for each node of the big 
tree the entry $l$ refers to the neighbour of this vertex in the corresponding core in $S_x$ (or 
zero, if there is no neighbour in $S_x$) and, symmetrically, $r$ refers to the neighbour in the core 
in $S_y$ (or again zero if there is no neighbour there). The brother and son entries specify the tree 
structure, and this is similar to what we described for small trees.

These trees are substantially bigger (hence the name). The number of leaves (nodes at the bottom 
level $8$) is $(6!)^2=720^2=518400$ for segments of type $(6,6)$, it is $6!\cdot 4!=720\cdot 24
=17280$ for segments of type $(6,4)$, and it is $(4!)^2=24^2=576$ for segments of type $(4,4)$. So it 
is really fortunate that we never need to deal with segments $S=S_z$ of type $(6,6)$, as the handle 
$X=S_y\cap S_x$ is always a non-edge in the final enumeration. The total size of the tree varies 
depending on the exact quad type.

Again, this convenient tree structure allows for a rather uncomplicated enumeration code allowing us 
to consider all possible matchings between $S_z$ and the other two segments in a single loop.

\section{Implementation of the LDLT algorithm} \label{appendix LDLT}

In this appendix, we provide details of the implementation of the LDLT algorithm. Generally, the 
purpose of this algorithm is to decompose a symmetric matrix $A$ as a product $A=LDL^T$, where $L$ is 
a lower unitriangular matrix and $D$ is diagonal. We use this to decide whether $A$ is semi-positive 
definite, namely, this is so when all (diagonal) entries of $D$ are non-negative. We apply this to 
the Gram matrix $A=M$ of the given set of vectors $T=S_x\cup S_y\cup S_z$, which must be 
semi-positive definite, or else we can eliminate this particular configuration. 

We use the iterative version of the LDLP algorithm\footnote{We developed our iterative version 
starting from the code kindly given to us by M. Whybrow.}, that is, we start with the empty set and 
then extend it by one vertex at a time. As we do it, we extend the current matrices $L$ and $D$ 
accordingly, by one dimension. Hence, at the moment when we deal with $k$ vectors from $T$, our 
matrices $L$ and $D$ are of size $k\times k$ and $LDL^T$ coincides with the principal $k\times k$ 
minor of $M$.

Whenever we discover that $M$ is not semi-positive definite, we interrupt right away and switch to 
the next configuration $T$. On the other hand, if $M$ turns out to be semi-positive definite, the 
general algorithm involves a further step, adding further vertices to the $30$ vertices of $T$. At 
this step, we use a different technique based on the projection map to the subspace spanned by $T$. 
The matrix $R=L^{-1}$ is needed to determine this map and we compute $R$ iteratively alongside $L$ 
and $D$.

Now that we explained what we are doing in this algorithm, we are ready to present the code.
\begin{verbatim}
#
# extending the LDLT decomposition by one dimension
# if semi positive definite
#
# based on Madeleine Whybrow's code
#

M:=List([1..30],i->List([1..30],j->0));
L:=List([1..40],i->List([1..30],j->0));
D:=[];

R:=List([1..30],i->List([1..30],j->0));
Id:=IdentityMat(30);

AddOne:=function(r)
  local n,i,sum,j;
  n:=Length(r);

  for i in [1..n] do
    sum:=0;
    for j in [1..i-1] do
      sum:=sum+L[n][j]*L[i][j]*D[j];
    od;
    if i<n then
      if D[i]=0 then
        if r[i]=sum then
          L[n][i]:=0;
        else
          return false;
        fi;
      else
        L[n][i]:=(r[i]-sum)/D[i];
      fi;
    else
      L[n][n]:=1;
      D[n]:=r[n]-sum;
      if D[n]<0 then
        return false;
      fi;
    fi;
  od;

  for i in [1..n] do
    M[n][i]:=r[i];
    M[i][n]:=r[i];
  od;
  R[n]:=Id[n];
  for j in [1..n-1] do
    R[n]:=R[n]-L[n][j]*R[j];
  od;

  return true;
end;
\end{verbatim}
Note that we build the Gram matrix $M$ alongside $L$, $D$, and $R=L^{-1}$, as we add each time the 
new row $r$ of $M$, which serves as the input to the function \verb|AddOne| and which also provides 
the next size $n$ of $L$, $D$, $M$, and $R$. Also note that we only keep the diagonal entries of $D$, 
so this is a $1$-dimensional array in the code. We treat all the outputs $L$, $D$, $R$, and $M$ as 
global variables, as we want to have an easy access to them from our main enumeration code. It also 
saves a bit of time as we do not pass them back and forth as arguments.

The function \verb|AddOne| returns true if the extended matrix $M$ is semi-positive definite and it 
returns false otherwise. Note that the latter can happen in two different ways. First, as we 
discussed above, the new entry in $D$ may be negative, which clearly means that $M$ is not 
semi-positive definite. The second possibility amounts to the algorithm being unable to construct the 
extended $L$ and $D$. We now show that this may only happen when $M$ is not semi-positive definite, 
and so we get our answer anyway.

Note that in this lemma we again treat $D$ as a diagonal matrix.

\begin{lemma}\label{LDLT non-SPD}
If $D_{ii}=0$, for $1\leq i<n$, and $r_i\not=\sum_{j=1}^{i-1}L_{kj}L_{ij}D_{jj}$ then $M$ 
is not semi-positive definite.
\end{lemma}

\begin{proof}
We start by reviewing the meaning of the matrices $L$ and $D$. Recall that $M$ is the Gram matrix of 
a subset of $T$ with respect to the dot product, but for extra generality, we view it simply as the 
Gram matrix of the standard basis $e_1,e_2,\ldots,e_n$ in $\RR^n$ with respect to a symmetric 
bilinear form $(\cdot,\cdot)$. Consider the linear map $\psi:\RR^n\to W=\la T\ra$ sending $e_i$ to 
$t_i$. Then the required symmetric bilinear form on $\RR^n$ is defined by $(e,f):=\psi(e)\cdot
\psi(f)$ for $e,f\in\RR^n$. Clearly, $M$ is the Gram matrix of $(\cdot,\cdot)$ with respect to the 
standard basis $e_1,e_2,\ldots,e_n$.

We now apply the Gram-Schmidt orthogonalisation process to the form $(\cdot,\cdot)$ to find an 
orthogonal basis $u_1,u_2,\ldots,u_n$ in $\RR^n$. Namely, $u_1=e_1$ and, inductively, 
$$u_i=e_i-\sum_{j=1}^{i-1}L_{ij}u_j,$$
for $i=1,2,\ldots,n$. Here 
$$
L_{ij}=\left\{
\begin{array}{cl}
0,&\mbox{if }(u_j,u_j)=0;\\
\frac{(e_i,u_j)}{(u_j,u_j)},&\mbox{otherwise}.\\
\end{array}
\right.
$$
Consequently, $e_i=\sum_{j=1}^iL_{ij}u_j$, where $L_{ii}=1$. This clarifies the meaning of the lower 
triangular matrix $L$; the diagonal entries from $D$ are given simply by $D_{ii}=(u_i,u_i)$ for all 
$i$. 

The basis $u_1,u_2,\ldots,u_n$ is indeed orthogonal provided that the form $(\cdot,\cdot)$ is positive 
definite or semi-positive definite. Under this condition, if $D_{ii}=(u_i,u_i)=0$ then the vector 
$u_i$ is in the radical of the form, i.e., it is orthogonal to the entire $\RR^n$. In particular, for 
$k>i$, we have $0=(e_k,u_i)=(e_k,e_i-\sum_{j=1}^{i-1}L_{ij}u_j)=(e_k,e_i)-\sum_{j=1}^{i-1}L_{ij}(e_k,
u_j)=M_{ki}-\sum_{j=1}^{i-1}L_{ji}L_{kj}D_{jj}$, since $(e_k,u_j)=L_{kj}D_{jj}$ for all $j\leq k$. 
This yields $M_{ki}=\sum_{j=1}^{i-1}L_{ji}L_{kj}D_{jj}$, and the contradiction shows that the form is 
not semi-positive definite. 

It remains to notice that when this is first encountered we have $n=k$ and $M_{ki}=r_i$.
\end{proof}

Hence indeed our function \verb|AddOne| returns false exactly when the extended Gram matrix $M$ is not 
semi-positive definite.

We would like to retain the map $\psi:\RR^n\to\la t_1,t_2,\ldots,t_n\ra$ introduced in this proof to 
be used elsewhere in the paper. Typically, it will be with the full set $T$, that is, with $n=30$. 
Then $u\in\RR^{30}$ can be viewed simply as the coefficient vector of $\psi(u)$ with respect to the 
spanning set $T=\{t_1,t_2,\ldots,t_{30}\}$ of $W=\la T\ra$. Since the dot product is positive definite 
on $W$, the vectors $u_i$ with $D_{ii}=(u_i,u_i)=0$ are in the kernel of $\psi$ and, in fact, such 
vectors $u_i$ form a basis of $\ker\psi$. On the other hand, the vectors $\psi(u_i)$, where 
$D_{ii}\neq 0$, form an orthogonal basis of $W$.

The final comment in this section concerns the meaning of the matrix $R=L^{-1}$. Note that we have 
$e_i=\sum_{j=1}^iL_{ij}u_j=\sum_{j=1}^nL_{ij}u_j$, since $L$ is lower triangular. This means that $L$ 
is the transition matrix from the orthogonal basis $\{u_1,u_2,\ldots,u_n\}$ to the standard basis 
$\{e_1,e_2,\ldots,e_n\}$ of $\RR^n$. Correspondingly, $R=L^{-1}$ is the transition matrix from the 
standard basis to the orthogonal basis $\{u_1,u_2,\ldots,u_n\}$. In other words, the row $R_i$ of $R$ 
provides the coefficients of $u_i$ with respect to the standard basis, or in simpler terms, 
$R_i=u_i$. 

\section{Projection matrix} \label{projection}

If the Gram matrix of the full set $T$ happens to be semi-positive definite then we have to consider 
vertices beyond $T$ and, at this stage, our strategy of checking semi-positive definiteness 
iteratively is not as effective, because we are not working with a pre-selected segment and, 
consequently, the graph structure on the additional set of vertices is not known. Hence we need 
another approach, and it involves computing the orthogonal projection of the additional vertices to 
the subspace $W$ spanned by $T$. For a vertex $u$, we will have the list of its neighbours in $T$ and 
hence we can form the vector $r=(r_1,r_2,\ldots,r_n)$, where $n=|T|=30$ and $r_i=u\cdot t_i$ is the 
value of the dot product between $u$ and the $i$th vertex $t_i\in T$. Hence
$$
\renewcommand{\arraystretch}{1.5}
r_i=\left\{
  \begin{array}{rl}
     \frac{2}{7},&\text{if }u\sim t_i,\\
     -\frac{1}{14},&\text{otherwise}
  \end{array}
\right.
$$
We need to find a matrix $P$ such that $p=rP$ is the list of coefficients of the projection 
$\proj_W(u)$ of $u$ to $W$ with respect to the spanning set $T=\{t_1,t_2,\ldots,t_n\}$. (That is, 
$\proj_W(u)=\psi(p)$, where $\psi:\RR^n\to W$ is the linear map introduced in Appendix \ref{appendix 
LDLT}.) We need to make two comments. First, since $T$ may in some cases be linearly dependent, the 
vector $p$ is in general not unique. Secondly, for the same reason, some vectors $r$ need to be 
eliminated because they lead to non semi-positive definiteness of the inner product on the expanded 
space $\la T\cup\{u\}\ra$. This happens when the entries in $r$ yield non-orthogonality of $u$ to some 
linear combinations of the vectors $t_i$ that have zero length. 

Here is our function:
\begin{verbatim}
# Vertex projection

P:=List([1..30],i->List([1..30],j->0));

ComputeProjMat:=function()
  local i;
  P:=0*IdentityMat(30);
  
  for i in [1..30] do
    if D[i]<>0 then
      P:=P+TransposedMat([R[i]])*[R[i]]/D[i];
    fi;
  od;
  
end;
\end{verbatim}
Note that we again treat the output, the projection matrix $P$, as a global variable, because we want 
to have easy access to it from the main code.

Here is the lemma that justifies our method of computing $P$. Recall that when we need $P$ we have 
already determined the matrices $L$ (strictly lower triangular) and $D$ (diagonal), such that $M=LDL^T$ 
is the Gram matrix of the set $T$. Furthermore, we also have $R=L^{-1}$. Let $u_i=R_i$ be the $i$th row 
of the matrix $R$. Recall that $u_1,u_2,\ldots,u_n$ is the orthogonal basis of $\RR^n$ with respect to 
the form $(\cdot,\cdot)$ with the Gram matrix $M$.

For generality, we allow an arbitrary $n=|T|$, but of course, in our application $n=30$. 
Correspondingly, all matrices are $n\times n$. Let $N=\left\{i\in\{1,2,\ldots,n\}\mid D_{ii}\neq 
0\right\}$.

\begin{lemma}\label{ProjVector}
For a vertex $u\notin T$ identified by the vector $r$, with $r_i=u\cdot t_i$, $i=1,2,\ldots,n$, the 
projection of $u$ onto $W=\la T\ra=\la t_1,t_2,\ldots,t_n\ra$ coincides with $w=\sum_{i=1}^n p_it_i$, 
where $p=(p_1,p_2,\ldots,p_n)=rP$ and 
$$
P=\sum_{i\in N}\frac{1}{D_{ii}}R_i^T R_i.
$$
\end{lemma}

\begin{proof}
The projection vector $w$ is identified by the property that $w\cdot v=u\cdot v$ for all $v\in W$. In 
particular, $w\cdot t_i=u\cdot t_i=r_i$, for $i=1,2,\ldots,n$. Transferring this into $\RR^n$, we are 
looking for a row vector $p$ (coefficients of $w$ with respect to the set $T$) such that $pM=r$. If $M$ 
is not positive definite, the set $T$ is linearly dependent, and so such a vector $p$ is not unique, 
but it is unique if we select it in the subspace $\la u_i\mid i\in N\ra$, which is a complement to the 
radical of the form $(\cdot,\cdot)$. We claim that the formulae in the lemma give us exactly such a 
vector.

Indeed, suppose that $p$ is a row vector such that $pM=r$ and $p$ is a linear combination of the vectors 
$u_i$, $i\in N$. Using that $u_i=R_i$, for all $i$, we get
\begin{align*} 
\displaystyle
p&=\sum_{i\in N}\frac{(p,u_i)}{(u_i,u_i)}u_i\\
 &=\sum_{i\in N}\frac{(p,R_i)}{(R_i,R_i)}R_i\\
 &=\sum_{i\in N}\frac{(p,R_i)}{D_{ii}}R_i\\
 &=\sum_{i\in N}\frac{1}{D_{ii}}(pMR_i^T)R_i\\
 &=pM\sum_{i\in N}\frac{1}{D_{ii}}R_i^T R_i\\
 &=r\sum_{i\in N}\frac{1}{D_{ii}}R_i^T R_i.
\end{align*}
Note that, in line 4 of this calculation, the product $pMR_i^T$ is a $1\times 1$ matrix with the entry 
equal to $(p,R_i)$, and so we have a four-term matrix product here, which allows us to use associativity 
and distributivity in the following line.
\end{proof}

Thus, we have the correct formula for the projection matrix.

\end{document}